\documentclass[11pt]{article}

\usepackage[a4paper,margin=1in]{geometry}
\usepackage{amsmath,amssymb,amsthm,mathtools}
\usepackage{enumitem}
\usepackage{hyperref}
\usepackage{microtype}
\usepackage{mathrsfs}

\hypersetup{colorlinks=true,linkcolor=blue,citecolor=blue,urlcolor=blue}

\newtheorem{theorem}{Theorem}[section]
\newtheorem{lemma}[theorem]{Lemma}
\newtheorem{proposition}[theorem]{Proposition}
\newtheorem{corollary}[theorem]{Corollary}

\theoremstyle{definition}
\newtheorem{definition}[theorem]{Definition}
\newtheorem{example}[theorem]{Example}
\theoremstyle{remark}
\newtheorem{remark}[theorem]{Remark}

\newcommand{\cC}{\mathcal{C}}

\newcommand{\cI}{\mathcal{I}}
\newcommand{\cP}{\mathcal{P}}
\newcommand{\Gen}{\operatorname{Gen}}
\newcommand{\diam}{\operatorname{diam}}

\newcommand{\Aut}{\operatorname{Aut}}

 \newcommand{\Nat}{\mathbb{N}}
 \newcommand{\Int}{\mathbb{Z}}

 \newcommand{\set}[1]{\left\{#1\right\}}
 \newcommand{\Set}[2]{\set{#1\ \vert\ #2}}
 
		\newcommand{\subjclass}[1]{%
  {\footnotesize
  \textit{2020 Mathematics Subject Classification:} \textbf{#1}}%
}

\newcommand{\keywords}[1]{%
  {\footnotesize
  \noindent\textit{Key words and phrases:} #1%
  }%
}

\title{Generalized Graph Compositions with Applications to Difference Graphs of Finite Groups}

\author{Shamik Ghosh,\thanks{Corresponding author}\ \thanks{Department of Mathematics, Jadavpur University, Kolkata-700032, India. \texttt{ghoshshamik@yahoo.com}}\ \ Sanchita Paul\thanks{Department of Commerce and Management, St.~Xavier's University, Kolkata - 700160, India.\\ \texttt{sanchitajumath@gmail.com}}\ \ and M.~K.~Sen\thanks{Department of Pure Mathematics, University of Calcutta, Kolkata - 700019, India. \texttt{senmk9@gmail.com}}}

\date{}

\begin{document}
\maketitle

\begin{abstract}
The difference graph $D(G)$ of a finite group $G$ is obtained from the
edge difference between its intersection power graph and power graph,
after deleting isolated vertices. This graph has already been studied,
with sufficient conditions for connectedness and a diameter bound $6$
for finite groups satisfying those conditions.

We use generalized graph composition to reduce $D(G)$ to a graph $B(G)$
on the cyclic subgroups of $G$, so that connectedness and diameter are
determined by the subgroup structure of $G$. We obtain a general
criterion for the non-emptiness of $B(G)$ in terms of branching
subgroups and b-normality, and characterize its connectedness for finite
$p$-groups, non-cyclic finite abelian groups, and non-abelian groups
with both trivial and non-trivial center. Combined with the previously
established cyclic-group case, this gives a complete characterization of
non-emptiness and connectedness of difference graphs for all finite groups.

The successive structural cases lead naturally to the sharp diameter bounds
$2,3,4,$ and $5$. For centerless non-abelian groups,
connectedness is governed either by a unique branching subgroup or by
an auxiliary graph $\mathcal A(G)$; in the latter case
\[
\operatorname{diam}\mathcal A(G)-1
\leq \operatorname{diam}B(G)
\leq
\max\{4,\operatorname{diam}\mathcal A(G)+1\},
\]
and both bounds are sharp.
\end{abstract}

\noindent

\noindent


\section{Introduction}

The lexicographic product is a classical graph product in which every
vertex of a graph is replaced by a copy of another graph. Its
heterogeneous version, in which different vertices may be replaced by
different graphs, is usually called a generalized composition and has
also appeared under such names as an $X$-join, an $H$-join, or a
generalized lexicographic product. The construction goes back to the
work of Sabidussi \cite{Sabidussi1959,Sabidussi1961} and has been developed
further in more recent work~\cite{CardosoGomesPinheiro2022, Gerbaud2010}. In this paper we use generalized
composition as a reduction device: a graph defined on the elements of
an algebraic structure is replaced by a substantially smaller graph on
its substructures, while preserving the information relevant to
connectedness and distance.

The study of graphs associated with groups and semigroups has a long history, with Bos$\acute{\text{a}}$k's 1964 work \cite{Bosak1964} as an early example. In 2009, Chakrabarty, Ghosh and Sen \cite{CSS2009} introduced the undirected power graph of a semigroup, a construction that in particular applies to groups. Cameron and Ghosh \cite{PCSG2011} subsequently studied the power graph of a finite group, including the isomorphism problem. Since then, power graphs and many of their variants have been studied extensively.
 Our application concerns the difference graph $D(G)$ of a finite group
$G$, obtained from the edge difference between the intersection power
graph and the power graph after deleting isolated vertices. The
intersection power graph was introduced by Bera \cite{Bera2018}, and Bera
and Cameron \cite{BeraCameron2025} subsequently initiated the study of
$D(G)$. They proved that, when $\pi(Z(G))\geq 2$, the graph $D(G)$ is
connected apart from the cyclic group of order $pq$, and obtained the
diameter bound $\operatorname{diam}D(G)\leq 6$. They left as their first
open connectivity question the case $\pi(Z(G))\leq1$, asking what can
be said about connectedness in this remaining situation.

In this paper, we give a complete characterization of the non-emptiness
of the difference graph for finite groups and, together with the
corresponding result for cyclic groups in \cite{BeraCameron2025}, a complete characterization of
its connectedness. In particular, we answer
the connectivity part of the first open question posed by Bera and Cameron
for groups with $\pi(Z(G))\leq1$, while at the same time replacing the
earlier diameter estimate by sharp bounds adapted to the different
structural classes.

Our starting point is to consider the graph $B(G)$ on the non-isolated
cyclic subgroups of $G$, where two cyclic subgroups are adjacent when
they have non-trivial intersection and are incomparable under inclusion.
We show that $D(G)$ is a null generalized composition over $B(G)$.
Consequently, $D(G)$ and $B(G)$ have the same connectedness behaviour,
and their diameters are directly related. This reduction allows the
problem to be treated through the cyclic-subgroup structure of $G$
rather than through individual elements.

Using the notions of branching prime-order subgroups and $b$-normality,
we obtain a general criterion for non-emptiness and connectivity results
covering all finite groups. For finite $p$-groups we obtain the sharp
diameter bound $2$; for non-cyclic finite abelian groups which are not
$p$-groups, the sharp bound $3$; for non-abelian groups whose center is
not a $p$-group, we improve the previous bound $6$ to the sharp bound
$4$; and when the center is a non-trivial $p$-group, we characterize
connectedness through centralizers of non-central branching subgroups
and obtain the sharp bound $5$.

For non-abelian groups with trivial center, connectedness is
governed by the branching prime-order subgroups and an auxiliary graph
$\mathcal A(G)$ recording commuting branching subgroups of different
prime orders. This yields a complete connectivity criterion and sharp
diameter estimates in terms of $\mathcal A(G)$. 

The paper is organized as follows. Section~2 records the facts about
generalized compositions needed later. Section 3 recalls the base graph $B(G)$ and expresses $D(G)$ as a null
generalized composition over it. Section~4 treats finite
$p$-groups and develops branching and $b$-normality. Section~5 deals
with finite abelian groups. Section~6 studies non-abelian groups with
non-trivial center, separating the cases in which the center is or is
not a $p$-group. Section~7 treats non-abelian groups with
trivial center and introduces the auxiliary graph $\mathcal A(G)$.
Section~8 contains concluding remarks.

For terminology in graph theory and group theory one may consult \cite{DBW} and \cite{Rotman} respectively. However, we include some necessary definitions and notations which are used frequently in the paper. All graphs in this paper are finite, simple, and undirected. We write $V(X)$ and $E(X)$ for the vertex and edge sets of a graph $X$,
respectively. For
vertices $x,y$ of a graph $X$, we write $x\sim_X y$, or simply
$x\sim y$ when the graph is clear, if $x$ and $y$ are adjacent. A graph is {\em complete} if every two distinct vertices are adjacent. We call
a graph {\em null} (or edgeless) if it has no edges, and {\em empty} if its vertex
set is empty. Thus a null graph may have vertices, whereas an empty
graph has none. 
A {\em walk} is a sequence of vertices in which consecutive vertices are
adjacent, and a {\em path} is a walk with no repeated vertices. A non-empty graph is
{\em connected} if every pair of distinct vertices are joined by a path; a maximal
connected subgraph is a {\em component}. A {\em clique} is a set of
pairwise adjacent vertices. The {\em distance} $d_X(x,y)$ is the length of a shortest $x$--$y$ path.
The {\em eccentricity} of a vertex $x$ is
\[
\operatorname{ecc}_X(x)=\max_{y\in V(X)}d_X(x,y),
\]
and the {\em diameter} of a connected graph $X$ is
\[
\operatorname{diam}X=\max_{x,y\in V(X)}d_X(x,y).
\]
The subgraph {\em induced} by
$S\subseteq V(X)$ has vertex set $S$ and contains all edges of $X$
with both endpoints in $S$. A graph is {\em bipartite} if its vertex set
can be partitioned into two sets, neither of which contains an edge.
The {\em line graph} $L(X)$ has the edges of $X$ as its vertices, two being
adjacent when the corresponding edges of $X$ have a common endpoint. $K_n$ and $K_{m,n}$ denote the complete and complete
bipartite graphs, and $\overline K_n$ the null graph on $n$ vertices.
A vertex having no neighbours is {\em isolated}, and $\dot\cup$
denotes disjoint union. A Cartesian product $X\square Y$ of two graphs $X=(V_1,E_1)$ and $Y=(V_2,E_2)$ has the vertex set $V_1\times V_2$ and two vertices $(a,b)$ and $(c,d)$ are adjacent if and only if either $a=c$ and $bd\in E_2$ or $ac\in E_1$ and $b=d$. The join $X\vee Y$ of graphs with disjoint vertex sets is obtained
from $X\dot\cup Y$ by adding all edges between $V_1$ and $V_2$.

We also recall the group-theoretic notation used throughout the paper.
All groups considered here are finite unless stated otherwise. The
identity element of a group $G$ is denoted by $e$, and the order of
$G$, of an element $x$, or of a subgroup $H$ is denoted by
$|G|$, $|x|$, or $|H|$, respectively. We write $H\leq G$ if $H$ is
a subgroup of $G$, and $H<G$ if it is a proper subgroup. The cyclic
subgroup generated by $x\in G$ is denoted by $\langle x\rangle$;
a group is {\em cyclic} if it is generated by a single element, and {\em abelian}
if all of its elements commute. 
For a prime $p$, a {\em $p$-group} is a group whose order is a power of
$p$. A {\em Sylow $p$-subgroup} of $G$ is a subgroup whose order is the
largest power of $p$ dividing $|G|$. We denote by
\[
\pi(G)=\{p:p\text{ is a prime divisor of }|G|\},
\]
the set of prime divisors of $|G|$. An element whose order is a power of $p$ is a {\em $p$-element};
the {\em exponent} of $G$ is the least positive integer $n$ such that
$g^n=e$ for every $g\in G$, and a {\em maximal cyclic subgroup} is a
cyclic subgroup maximal under inclusion. We write $\operatorname{Aut}(G)$ for the automorphism group of $G$. 
For $H\leq G$ and $g\in G$, the {\em conjugate subgroup} of $H$ by $g$ is
\[
H^g=gHg^{-1}.
\]
A subgroup $H$ is {\em normal} in $G$, written $H\trianglelefteq G$, if
$H^g=H$ for every $g\in G$. It is {\em characteristic} in $G$ if
$\alpha(H)=H$ for every automorphism $\alpha$ of $G$. The {\em center} of
$G$, the {\em centralizer} of $H$ in $G$, and the {\em normalizer} of $H$ in $G$
are, respectively,
\[
Z(G)=\{g\in G:gx=xg\text{ for every }x\in G\},
\]
\[
C_G(H)=\{g\in G:gh=hg\text{ for every }h\in H\},\ \text{ and }\ N_G(H)=\{g\in G:H^g=H\}.
\]
For an element $x\in G$, we write \(
C_G(x)=\{g\in G:gx=xg\}=C_G(\langle x\rangle)
\) 
for the {\em centralizer} of $x$ in $G$.
An element of order $2$ is an {\em involution}. We use
$G_1\times G_2$ for a direct product and $N\rtimes H$ for a
semidirect product, with the relevant action specified when needed. We denote the set $\set{1,2,\ldots,n}$ by $[n]$ for any $n\in\Nat$.


\section{Generalized compositions of graphs}
\label{sec:composition}

Let
$G$ be a graph with
\[
 V(G)=\{v_1,\ldots,v_n\},\qquad n\geq 2,
\]
and let $H_1,\ldots,H_n$ be pairwise vertex-disjoint nonempty graphs. The \emph{generalized composition} of $G$ by $H_1,\ldots,H_n$, denoted by
\[
 X=G[H_1,\ldots,H_n],
\]
is the graph with vertex set
\[
 V(X)=\bigcup_{i=1}^{n}V(H_i),
\]
where, for $x\in V(H_i)$ and $y\in V(H_j)$, the vertices $x$ and $y$ are
adjacent in $X$ if and only if either
\begin{enumerate}[label=\textup{(\roman*)}]
 \item $i=j$ and $xy\in E(H_i)$, or
 \item $i\neq j$ and $v_iv_j\in E(G)$.
\end{enumerate}
The graph $G$ is the \emph{base graph}, and the graphs $H_i$ are
the \emph{fibres} or \emph{substituted graphs}.

\vspace{1em}
If $H_1\cong\cdots\cong H_n\cong H$, then the construction reduces to the
ordinary {\em lexicographic product} $G[H]$. Let \(X\) be a graph. A nonempty set \(M\subseteq V(X)\) is a
\emph{module} of \(X\) if every vertex of \(V(X)\setminus M\)
is adjacent either to all vertices of \(M\) or to none of them. Equivalently, all vertices of $M$ have the same neighborhood outside $M$ \cite{HabibPaul2010}. 

\begin{proposition}\label{prop:modules}
For every \(i\in [n]\), the set \(V(H_i)\) is a module of
\(X=G[H_1,\ldots,H_n]\). Moreover, contracting each \(V(H_i)\) to a single
vertex yields a graph isomorphic to \(G\).
\end{proposition}

\begin{proof}
Every vertex outside $V(H_i)$ belongs to some $H_j$, $j\neq i$, and is adjacent
either to all of $V(H_i)$ or to none of it, according as $v_iv_j$ is or is not
an edge of $G$.  The second assertion follows immediately from the definition.
\end{proof}

It is worth noting that the connectedness of a composition depends
only on the connectedness of the base graph and not on the substituted
graphs.

\begin{theorem}[Connectedness transfer]\label{thm:composition-connected}
The graph $X=G[H_1,\ldots,H_n]$ is connected if and only if $G$ is connected.
\end{theorem}

\begin{proof}
Suppose that $G$ is connected.  Let $x\in V(H_i)$ and $y\in V(H_j)$.  If $i\neq j$, a path
\[
 v_i=v_{k_0},v_{k_1},\ldots,v_{k_t}=v_j
\]
in $G$ lifts to a path in $X$ by choosing one vertex from every corresponding
fibre, beginning with $x$ and ending with $y$.  If $i=j$ and $x\neq y$, choose
a neighbour $v_r$ of $v_i$ in $G$ and a vertex $z\in V(H_r)$; then $x-z-y$ is a path in $X$ 
of length two.  Thus $X$ is connected.

Conversely, if $G$ is disconnected, the union of the fibres indexed by one
component of $G$ has no edge to the union of the fibres indexed by another
component.  Hence $X$ is disconnected.
\end{proof}

\begin{lemma}[Distance between distinct fibres]\label{lem:distance}
Assume that $G$ is connected.  If $i\neq j$, $x\in V(H_i)$, and
$y\in V(H_j)$, then
\[
 d_X(x,y)=d_G(v_i,v_j).
\]
\end{lemma}

\begin{proof}
A shortest $v_i$--$v_j$ path in $G$ lifts to an $x$--$y$ path in $X$, so
\[
 d_X(x,y)\leq d_G(v_i,v_j).
\]
Conversely, let $x=x_0,x_1,\ldots,x_t=y$ be a shortest path in $X$, with
$x_s\in V(H_{k_s})$.  Project this path to the sequence
$v_{k_0},\ldots,v_{k_t}$ and delete consecutive repetitions.  Every remaining
step is an edge of $G$, so the resulting sequence is a $v_i$--$v_j$ walk of
length at most $t$.  Therefore
\[
 d_G(v_i,v_j)\leq t=d_X(x,y).
\]

\vspace{-2em}
\end{proof}

\begin{theorem}[Diameter formula]\label{thm:diameter}
Let \(X=G[H_1,\ldots,H_n]\), where \(G\) is connected. Then
\[
\operatorname{diam}(X)=
\begin{cases}
1, & \text{if \(G\) and every \(H_i\) are complete},\\[2mm]
2, & \text{if \(G\) is complete and some \(H_i\) is not complete},\\[2mm]
\operatorname{diam}(G), & \text{if \(G\) is not complete}.
\end{cases}
\]
\end{theorem}

\begin{proof}
By the preceding lemma, distances between vertices belonging to
distinct fibres are precisely the corresponding distances in \(G\).
If two distinct vertices \(x,y\) belong to the same fibre \(V(H_i)\),
then \(d_X(x,y)=1\) when \(xy\in E(H_i)\); otherwise
\(d_X(x,y)=2\), since \(v_i\) has a neighbour in the connected graph
\(G\) and every
neighbouring fibre is joined completely to $H_i$.  Thus the stated formula follows.
\end{proof}


\section{Difference graphs of finite groups}
\label{sec:difference}

Let $G$ be a finite group with identity $e$.  The \emph{power graph}
$\cP(G)$ has vertex set $G$, with two distinct elements adjacent when one is a
power of the other.  In the \emph{intersection power graph} $\cI(G)$, two
distinct non-identity elements $x,y$ are adjacent when 
\(
 \langle x\rangle\cap\langle y\rangle\neq\{e\}.
\) 
The conventional treatment of the identity is immaterial for the difference
below, since its incident edges occur in both graphs.

The \emph{difference graph} $D(G)$ is obtained from the graph with edge set
$E(\cI(G))\setminus E(\cP(G))$ by deleting all isolated vertices.

Let $\cC(G)$ denote the set of cyclic subgroups of $G$.  Define a graph
$\Gamma(G)$ on $\cC(G)$ by
\[
 H\sim K
 \quad\Longleftrightarrow\quad
 H\nsubseteq K,\ K\nsubseteq H,
 \text{ and }H\cap K\neq\{e\}.
\]
and let $B(G)$ be obtained from $\Gamma(G)$ by deleting its isolated vertices. For a cyclic subgroup $H$, let
\[
 \Gen(H)=\{x\in H:\langle x\rangle=H\}.
\]
Then $|\Gen(H)|=\varphi(|H|)$, where $\varphi$ denotes Euler's totient function.

The following proposition places the replacement construction in \cite{BeraCameron2025} explicitly
in the framework of generalized composition.

\begin{proposition}[Null-composition representation]\label{prop:null-composition}
If $B(G)=\varnothing$, then $D(G)=\varnothing$. If $B(G)\neq\varnothing$,
then the difference graph $D(G)$ is the generalized composition obtained from
$B(G)$ by replacing each vertex (subgroup) $H$ by the null graph on $\Gen(H)$.  In
particular,
\[
 D(G)\cong
 B(G)\bigl[\,\overline{K}_{\varphi(|H|)}:H\in V(B(G))\,\bigr].
\]
\end{proposition}

\begin{proof}
If $B(G)=\varnothing$, then there are no pairs of distinct cyclic
subgroups satisfying the adjacency condition, and hence
$D(G)=\varnothing$. Assume therefore that $B(G)\neq\varnothing$.
Since $B(G)$ has no isolated vertices, it has at least two vertices,
so the generalized composition of Section~2 is defined.

Elements generating the same cyclic subgroup are adjacent in the power graph. So, no two of them are adjacent in $D(G)$.  Let $x\in\Gen(H)$ and
$y\in\Gen(K)$ with $H\neq K$.  They are adjacent in the intersection power
graph exactly when $H\cap K\neq\{e\}$, and they are adjacent in the power graph
exactly when $H\subseteq K$ or $K\subseteq H$.  Therefore $xy$ is an edge of
$D(G)$ exactly when $H$ and $K$ are adjacent in $B(G)$.  The cross adjacency
between two generator classes is consequently complete or empty according to
the corresponding edge of $B(G)$.
\end{proof}

\begin{corollary}\label{cor:D-transfer}
The graph $D(G)$ is connected if and only if $B(G)$ is connected.  If $B(G)$
is connected, then
\[
 \diam D(G)=\max\{2,\diam B(G)\}.
\]
In particular, if $B(G)$ is connected and non-complete, then
\[
 \diam D(G)=\diam B(G).
\]
\end{corollary}

\begin{proof}
Connectedness follows from Theorem~\ref{thm:composition-connected}.  A cyclic
subgroup of prime order is isolated in $\Gamma(G)$: any non-trivial
intersection with it forces containment.  Hence every vertex (subgroup) $H$ of $B(G)$ has
composite order and $\varphi(|H|)\geq2$.  Every fibre in
Proposition~\ref{prop:null-composition} is therefore non-complete, and the
diameter assertion follows from Theorem~\ref{thm:diameter}.
\end{proof}


\section{Finite \texorpdfstring{$p$}{p}-groups and branching subgroups}
\label{sec:pgroups}

Throughout this section, $G$ is a finite $p$-group. For an order-$p$ subgroup
$P\leq G$, set
\[
 \cC_P(G)=\{H\in\cC(G):P\leq H\}.
\]
Every non-trivial cyclic subgroup of $G$ contains a unique subgroup of order
$p$, so the sets $\cC_P(G)$ partition the non-trivial cyclic subgroups.

\begin{lemma}[The tower above an order-$p$ subgroup]\label{lem:tower}
Let $P$ and $Q$ be distinct subgroups of $G$ of order $p$.
\begin{enumerate}[label=\textup{(\roman*)}]
 \item $\cC_P(G)\cap\cC_Q(G)=\varnothing$.
 \item No vertex from $\cC_P(G)$ is adjacent to a vertex from
       $\cC_Q(G)$ in $B(G)$.
 \item For $H,K\in\cC_P(G)$, one has $H\sim K$ in $B(G)$ if and only if
       $H$ and $K$ are incomparable under inclusion.
 \item The Hasse diagram of $(\cC_P(G),\subseteq)$ is a rooted tree with root
       $P$.
\end{enumerate}
\end{lemma}

\begin{proof}
The first assertion follows from uniqueness of the order-$p$ subgroup in a
cyclic $p$-group.  If $H\in\cC_P(G)$ and $K\in\cC_Q(G)$ had non-trivial
intersection, the cyclic group $H\cap K$ would contain both $P$ and $Q$, again
contradicting uniqueness.  This proves (ii), while (iii) follows because every
pair in $\cC_P(G)$ already has non-trivial intersection containing $P$.

For (iv), suppose that incomparable $H,K\in\cC_P(G)$ had a common cyclic
overgroup $L$.  Assume $|H|\leq |K|$.  The cyclic group $L$ has a unique
subgroup of order $|H|$, and therefore the subgroup of order $|H|$ inside $K$
must be $H$.  Hence $H\leq K$, a contradiction.  Thus incomparable elements
have no common upper bound, which gives the rooted-tree structure.
\end{proof}

\begin{definition}[Branching subgroup]
An order-$p$ subgroup $P$ of $G$ is called \emph{branching} if the rooted tree
associated with $\cC_P(G)$ is not a path.  Equivalently, $P$ is contained in at
least two distinct maximal cyclic subgroups of $G$.
\end{definition}

\begin{theorem}[Component theorem]\label{thm:components}
The number of connected components of $B(G)$ is equal to the number of
branching subgroups of order $p$.  Consequently:
\begin{enumerate}[label=\textup{(\roman*)}]
 \item $B(G)$ is empty if and only if no order-$p$ subgroup is branching;
 \item $B(G)$ is connected and non-empty if and only if exactly one order-$p$
       subgroup is branching;
 \item every connected component of $B(G)$ has diameter at most two.
\end{enumerate}
\end{theorem}

\begin{proof}
By Lemma~\ref{lem:tower}, distinct sets $\cC_P(G)$ have no edges between them.
If the tree above $P$ is a path, every two members of $\cC_P(G)$ are comparable,
so all of them are isolated in $\Gamma(G)$.  If the tree is not a path, its
non-isolated vertices form one component.  Indeed, let $H$ and $K$ be two such
vertices.  If they are incomparable, then they are adjacent.  If, say,
$H\subset K$, choose $L$ incomparable with $H$.  The subgroup $L$ is also
incomparable with $K$: neither $L\subseteq K$ nor $K\subseteq L$ is possible,
since the subgroups of a cyclic p-group are linearly ordered by inclusion.
 Hence $H\sim L\sim K$.  This also proves the diameter bound.
\end{proof}

The following criterion removes the reference to maximal cyclic subgroups.

\begin{lemma}[Elementwise branching criterion]\label{lem:elementwise}
Let $P\leq G$ have order $p$.  Then $P$ is branching if and only if there exist
$x,y\in G$ such that
\[
 |x|=|y|>p,\qquad \langle x\rangle\neq\langle y\rangle,
\]
and
\[
 e\neq x^{|x|/p}=y^{|y|/p}\in P.
\]
\end{lemma}

\begin{proof}
Suppose that $P$ is branching and choose distinct maximal cyclic subgroups
$H,K$ containing $P$.  Write $|H|=p^m$ and $|K|=p^n$, with $m\leq n$ after
interchanging them if necessary.  Choose a generator $x$ of $H$, and let $L$
be the unique subgroup of $K$ of order $p^m$.  Choose a generator $y_0$ of $L$.
Then $\langle x\rangle\neq\langle y_0\rangle$, for otherwise $H=L\leq K$,
contrary to the maximality and distinctness of $H$ and $K$.  Both cyclic groups
have $P$ as their unique subgroup of order $p$.  

Since \(H\) and \(L\) both
contain \(P\), and a cyclic group of order \(p^m\) has a unique
subgroup of order \(p\), we have
\[
P=\langle x^{p^{m-1}}\rangle
 =\langle y_0^{p^{m-1}}\rangle.
\]
Thus
\[
y_0^{p^{m-1}}
   =\bigl(x^{p^{m-1}}\bigr)^a
\]
for some \(a\) with \(p\nmid a\). Choose \(b\) such that
\(ab\equiv1\pmod p\), and put \(y=y_0^b\). Then \(y\) is again a
generator of \(L\), and
\[
y^{p^{m-1}}=x^{p^{m-1}}\neq e.
\]
The common element has order \(p\), and hence generates \(P\).

Conversely, let $x,y$ satisfy the displayed conditions and extend
$\langle x\rangle$ and $\langle y\rangle$ to maximal cyclic subgroups $H$ and
$K$.  They cannot be equal, since a cyclic group has a unique subgroup of each
order.  Their common order-$p$ subgroup is $P$, so $P$ is branching.
\end{proof}

We now isolate the normal branching subgroups inside an arbitrary $p$-group.

\begin{definition}[$b$-normal subgroup]\label{def:bnormal}
Let $L$ be a finite $p$-group and let $P\leq L$ have order $p$.  We say that
$P$ is \emph{$b$-normal in $L$} if $P\trianglelefteq L$ and at least one of the
following conditions holds:
\begin{enumerate}[label=\textup{(B\arabic*)}]
 \item $P$ is contained in a non-normal cyclic subgroup of $L$ of order
       greater than $p$;
 \item there exist distinct normal cyclic subgroups $H,K\trianglelefteq L$
       such that
       \[
       |H|=|K|>p\qquad\text{and}\qquad P\leq H\cap K.
       \]
\end{enumerate}
\end{definition}

\begin{theorem}[Normalizer characterization]\label{thm:normalizer}
Let $P\leq G$ have order $p$.  Then $P$ is branching in $G$ if and only if $P$ is $b$-normal in $N_G(P)=\Set{g\in G}{gPg^{-1}=P}$, the normalizer of $P$ in $G$.
\end{theorem}

\begin{proof}
We first note that every cyclic subgroup \(H\) containing \(P\) lies in \(N_G(P)\). Indeed, \(H\) is abelian, and hence every element of \(H\) centralizes \(P\). Thus \(H\leq C_G(P)\leq N_G(P)\), where $C_G(P)=\Set{z\in G}{zx=xz\text{ for all }x\in P}$. Hence the cyclic overgroups of $P$ in $G$ and in $N=N_G(P)$ are the same. Consequently, \(P\) branches in \(G\) if and only if it branches in \(N\). Moreover, \(P\trianglelefteq N\) by the definition of the normalizer.

Suppose that $P$ branches.  By Lemma~\ref{lem:elementwise}, there are distinct
cyclic subgroups $H,K\leq N$ of the same order greater than $p$, both containing
$P$.  If either is non-normal in $N$, condition \textup{(B1)} holds.  If both
are normal, condition \textup{(B2)} holds.  Thus $P$ is $b$-normal in $N$.

Conversely, condition \textup{(B2)} gives branching immediately.  Under
\textup{(B1)}, let $P\leq H$ with $H$ cyclic and non-normal in $N$.  Choose
$g\in N$ with $H^g=gHg^{-1}\neq H$.  Since $P\trianglelefteq N$, both $H$ and $H^g$
contain $P$, and they have the same order.  Hence $P$ branches.
\end{proof}

We shall use the fact that every normal subgroup \(P\) of order \(p\) in a finite \(p\)-group
\(G\) is central \cite[Chapter~4, Exercise~4]{Rotman}. Indeed, since conjugation induces a homomorphism \(G\longrightarrow \operatorname{Aut}(P)\), its image is a \(p\)-group, whereas \(|\operatorname{Aut}(P)|=p-1\). Hence the image is trivial. Therefore every element of \(G\) centralizes \(P\), and so \(P\leq Z(G)\), as required.

\begin{theorem}[Connectivity through $b$-normality]\label{thm:bnormal-connectivity}
For a finite $p$-group $G$, the following are equivalent:
\begin{enumerate}[label=\textup{(\roman*)}]
 \item $B(G)$ is connected and non-empty;
 \item exactly one subgroup $P$ of order $p$ is branching;
 \item exactly one subgroup $P$ of order $p$ is $b$-normal in its normalizer
       $N_G(P)$.
\end{enumerate}
Moreover, the unique subgroup $P$ is a characteristic subgroup of $G$, and hence
\[
 P\trianglelefteq G\qquad\text{and}\qquad P\leq Z(G).
\]
When these equivalent conditions hold,
\[
\operatorname{diam} B(G)\leq 2
\qquad\text{and}\qquad
\operatorname{diam} D(G)=2.
\]
\end{theorem}

\begin{proof}
The equivalence of (i) and (ii) is Theorem~\ref{thm:components}; the equivalence
of (ii) and (iii) is Theorem~\ref{thm:normalizer}. Let $P$ be the unique subgroup satisfying (iii), and let
$\alpha\in\Aut(G)$.  Since
\[
 N_G(\alpha(P))=\alpha(N_G(P)),
\]
and automorphisms preserve every condition in Definition~\ref{def:bnormal}, the
subgroup $\alpha(P)$ is $b$-normal in its normalizer.  Uniqueness gives
$\alpha(P)=P$.  Thus \(P\) is a characteristic subgroup of \(G\), and hence is normal
in \(G\), since it is invariant under all inner automorphisms. Then by the above mentioned fact, we have $P\leq Z(G)$. The diameter assertions follow from Theorem~\ref{thm:components} and Corollary \ref{cor:D-transfer}.
\end{proof}

\begin{example}
Let
\[
G=Q_8\times Q_8,
\qquad
Q_8=\{\pm1,\pm i,\pm j,\pm k\}\quad\text{with}\quad i^2=j^2=k^2=ijk=-1.
\]
The subgroups of order \(2\) in \(G\) are
\[
P_1=\langle(-1,1)\rangle,\qquad
P_2=\langle(1,-1)\rangle,\qquad
P_3=\langle(-1,-1)\rangle.
\]
All three are contained in \(Z(G)=\set{(1,1),(1,-1),(-1,1),(-1,-1)}\).

Since \(G\) has exponent \(4\), every nontrivial cyclic subgroup has
order \(2\) or \(4\). Each cyclic subgroup of order \(4\) contains
exactly one of \(P_1,P_2,P_3\), namely its unique subgroup of order
\(2\). Moreover, each \(P_i\) is branching. For example,
\[
P_1\leq \langle(i,1)\rangle\cap \langle(j,1)\rangle\qquad 
P_2\leq \langle(1,i)\rangle\cap \langle(1,j)\rangle
\quad
\text{ and }\quad 
P_3\leq \langle(i,i)\rangle\cap \langle(j,j)\rangle.
\]
Thus, by Theorem~\ref{thm:components}, \(B(G)\) has three connected components.

More precisely, the orders of the components can be determined explicitly. There are
\(6\cdot2=12\) elements \(x\in G\) satisfying
\[
x^2=(-1,1),
\]
and hence \(12/2=6\) cyclic subgroups of order \(4\) containing
\(P_1\), since each such subgroup has two generators. Similarly,
there are \(6\) cyclic subgroups of order \(4\) containing \(P_2\).
Finally, there are \(6\cdot6=36\) elements satisfying
\[
x^2=(-1,-1),
\]
giving \(36/2=18\) cyclic subgroups of order \(4\) containing \(P_3\).

Distinct cyclic subgroups of order \(4\) containing the same \(P_i\)
are incomparable and have nontrivial intersection, and hence are
adjacent in \(B(G)\). Therefore
\[
B(Q_8\times Q_8)
   \cong K_6\dot\cup K_6\dot\cup K_{18}.
\]

In fact, all three subgroups are \(b\)-normal in \(G\).  For
\(P_1\), the distinct normal cyclic subgroups
\[
\langle(i,1)\rangle,\qquad \langle(j,1)\rangle
\]
have order \(4\) and both contain \(P_1\); hence \(P_1\) satisfies
{\rm (B2)}.  Similarly, \(P_2\) satisfies {\rm (B2)} using
\[
\langle(1,i)\rangle,\qquad \langle(1,j)\rangle.
\]
On the other hand,
\[
P_3\leq \langle(i,i)\rangle,
\]
and \(\langle(i,i)\rangle\) is not normal in \(G\), since conjugation
by \((j,1)\) sends \((i,i)\) to \((-i,i)\notin\langle(i,i)\rangle\).
Thus \(P_3\) satisfies {\rm (B1)}.

So, here all three order-\(2\) subgroups are central and \(b\)-normal,
and \(B(G)\) has three connected components. 
\end{example}

The above example shows that centrality of the relevant order-\(p\)
subgroups does not imply connectedness. The following example further
shows that even the uniqueness of a central branching subgroup is not
sufficient for connectedness. We denote a cyclic group of order $n\in\Nat$ by $C_n$.

\begin{example}
Let
\[
G=(C_4\times C_4)\rtimes C_2
 =\langle a,b,t\mid
 a^4=b^4=t^2=e,\ ab=ba,\ tat=b,\ tbt=a\rangle .
\]
Thus conjugation by \(t\) interchanges \(a\) and \(b\). Elements of $G$ are either of the form $a^rb^s$ or $a^rb^st$, where $0\leq r,s<4$. Now an element \(a^r b^s\) commutes with $t$ if and only if $r=s$, while no element of the form $a^rb^st$ commutes with $a$ or $b$. Hence the center of \(G\) is
\[
Z(G)=\langle ab\rangle\cong C_4.
\]
Consequently,
\[
P=\langle a^2b^2\rangle
\]
is the unique central subgroup of order \(2\). The subgroup \(P\) is branching.  Indeed,
\[
(ab)^2=(ab^{-1})^2=a^2b^2,
\qquad
\text{while}
\qquad
\langle ab\rangle\neq\langle ab^{-1}\rangle.
\]
Hence Lemma~\ref{lem:elementwise} shows that \(P\) is branching.

However, there are also non-central branching subgroups of order \(2\).
For example, put
\[
P_1=\langle a^2\rangle,\qquad
P_2=\langle b^2\rangle.
\]
Since
\[
a^2=(ab^2)^2
\qquad\text{and}\qquad
b^2=(a^2b)^2
\qquad
\text{with}
\qquad
\langle a\rangle\neq\langle ab^2\rangle,
\qquad
\langle b\rangle\neq\langle a^2b\rangle,
\]
both \(P_1\) and \(P_2\) are branching by Lemma~\ref{lem:elementwise}.
They are non-central, since conjugation by \(t\) interchanges them:
\(
P_1^t=tPt^{-1}=P_2\neq P_1.
\)
\end{example}

Thus \(G\) has a unique central branching subgroup of order \(2\),
but also has non-central branching subgroups.  Hence \(B(G)\) is not
connected.  This shows that even uniqueness among the \emph{central}
branching subgroups is not sufficient; the condition in
Theorem~\ref{thm:bnormal-connectivity} is uniqueness among \emph{all} \(b\)-normal
(or, equivalently, branching) subgroups of order \(p\).

\begin{remark}
If we consider a stronger condition that \(G\) have a
unique subgroup \(P\) of order \(p\), then by the classical characterization
of finite \(p\)-groups with a unique subgroup of order \(p\), \(G\) is
either cyclic, or \(p=2\) and \(G\) is generalized quaternion \cite{ConradQuaternion}.

If \(G\) is cyclic, then its cyclic subgroups are linearly ordered by
inclusion, and hence \(B(G)\) is empty.  On the other hand, if
\(G=Q_{2^n}\), \(n\geq3\), with presentation
\[
Q_{2^n}
 =\langle a,b\mid
 a^{2^{n-1}}=e,\quad
 b^2=a^{2^{n-2}},\quad
 b^{-1}ab=a^{-1}\rangle,
\]
then
\[
P=\langle a^{2^{n-2}}\rangle
\]
is the unique subgroup of order \(2\).  Moreover,
\[
x=a^{2^{n-3}}
\qquad\text{and}\qquad
y=b
\]
have order \(4\), generate distinct cyclic subgroups, and satisfy
\[
x^2=y^2=a^{2^{n-2}}\neq e.
\]
Hence \(P\) is branching by Lemma~\ref{lem:elementwise}.  Therefore, by Theorem~\ref{thm:bnormal-connectivity},
\(B(Q_{2^n})\) is connected and nonempty.

More precisely, there are $2^{n-2}$ cyclic subgroups of order $4$
outside $\langle a\rangle$, and they are pairwise adjacent in $B(G)$.
The $n-2$ cyclic subgroups of $\langle a\rangle$ of orders
$4,8,\ldots,2^{n-1}$ form a chain and hence induce
$\overline{K}_{n-2}$. Every subgroup of the first family is adjacent
to every subgroup of the second. Therefore
\[
B(Q_{2^n})\cong K_{2^{n-2}}\vee\overline{K}_{n-2},
\]
where $\vee$ denotes the join of graphs.

Thus, among finite \(p\)-groups having a unique subgroup of order \(p\),
the graph \(B(G)\) is empty in the cyclic case and connected in the case of (non-abelian)
generalized quaternions.
\end{remark}


\section{Finite abelian groups}
\label{sec:abelian}

The cyclic case has already been determined by Bera and Cameron
\cite{BeraCameron2025}: if $G$ is cyclic, then $B(G)$ is empty precisely when
$\pi(G)\le1$ or $G\cong \mathbb Z_{pq}$ for distinct primes $p,q$;
in all other cases $B(G)$ is connected with diameter at most $3$, and
the bound is sharp. We therefore concentrate here on the additional
structure arising for non-cyclic finite abelian groups.

We first obtain the complete connectedness classification for abelian $p$-groups. As per the general convention of abelian group theory, in this section, we consider the binary operation of an abelian group as addition and denote the identity by $0$. Moreover, we denote a cyclic group of order $n\in\Nat$ by $\Int_n$. For a direct product of $m\in\Nat$ copies of a cyclic group $\Int_n$, we denote $\Int_n^m=(\Int_n)^m$.

\begin{theorem}[Abelian $p$-groups]\label{thm:abelian-p}
Let
\[
 G\cong \Int_{p^{n_1}}\times\cdots\times \Int_{p^{n_k}},
 \qquad 1\leq n_1\leq n_2\leq \cdots\leq n_k,\ k\in\Nat.
\]
Then:
\begin{enumerate}[label=\textup{(\roman*)}]
 \item $B(G)$ is empty if and only if $G$ is cyclic or elementary abelian;
 \item $B(G)$ is connected if and only if
       \[
       G\cong \Int_p^{\,m}\times \Int_{p^n}
       \qquad(m\geq1,\ n>1).
       \]
In every connected case, $\diam B(G)\leq2$ and $\diam D(G)=2$.
\end{enumerate}
\end{theorem}

\begin{proof}
If $G$ is cyclic, its cyclic subgroups form a chain, so $B(G)$ is empty.  If
$G$ is elementary abelian (i.e., $n_i=1$ for all $i\in [k]$), then every non-trivial cyclic subgroup has order $p$ and
is isolated, so again $B(G)$ is empty.

Assume now that exactly one invariant factor has exponent greater than $p$.
Then
\[
 G\cong \Int_p^{\,m}\times \Int_{p^n},\qquad m\geq1, n>1,
\]
and every cyclic subgroup of order greater than $p$ has the same order-$p$
subgroup
\[
 P=\{0\}\times\langle p^{n-1}\rangle.
\]
No other order-$p$ subgroup can lie in a cyclic subgroup of order greater than
$p$, because multiplication by $p$ annihilates the $\Int_p^m$ coordinates.  The
subgroup $P$ branches: if $u$ is a non-zero vector in $\Int_p^m$, then
\[
 \langle(0,1)\rangle
 \quad\text{and}\quad
 \langle(u,1)\rangle
\]
are distinct cyclic subgroups of order $p^n$ with the same order-$p$ subgroup
$P$.  Hence $P$ is the unique branching subgroup, and $B(G)$ is connected.

Finally, suppose that at least two invariant factors have exponent
greater than \(p\). For each \(i\), let \(e_i\) denote a fixed
generator of the \(i\)-th direct factor \(\Int_{p^{n_i}}\), regarded as
an element of \(G\) with all other coordinates zero. Thus
\[
|e_i|=p^{n_i},
\]
and the unique subgroup of order \(p\) in \(\langle e_i\rangle\) is
\[
P_i=\langle p^{\,n_i-1}e_i\rangle .
\]

Choose distinct indices \(i,j\) with \(n_i,n_j>1\). Consider
\[
x=e_i
\qquad\text{and}\qquad
y=e_i+p^{\,n_j-1}e_j .
\]
The element \(p^{\,n_j-1}e_j\) has order \(p\). Hence
\[
|y|=\operatorname{lcm}(p^{n_i},p)=p^{n_i}=|x|.
\]
Moreover,
\[
\langle x\rangle\neq\langle y\rangle,
\]
since every element of \(\langle x\rangle\) has zero \(j\)-th
coordinate, whereas \(y\) has nonzero \(j\)-th coordinate.

Finally,
\[
p^{\,n_i-1}y
 =p^{\,n_i-1}e_i+
   p^{\,n_i+n_j-2}e_j
 =p^{\,n_i-1}e_i,
\]
because \(n_i\geq2\) and \(e_j\) has order \(p^{n_j}\). Therefore
\[
P_i
 =\langle p^{\,n_i-1}x\rangle
 =\langle p^{\,n_i-1}y\rangle .
\]
Thus \(P_i\) is contained in two distinct cyclic subgroups of order
\(p^{n_i}\), and hence \(P_i\) is branching.

Interchanging \(i\) and \(j\) we can show that \(P_j\) is also branching.
Thus at least two order-\(p\) subgroups branch, and \(B(G)\) is
disconnected by Theorem~\ref{thm:components}. The diameter
statements follow from Theorem~\ref{thm:components} and
Corollary~\ref{cor:D-transfer}.
\end{proof}

We next consider non-cyclic abelian groups involving at least two primes.

\begin{theorem}\label{thm:abelian-mixed}
Let $G$ be a finite abelian group which is neither a $p$-group nor a cyclic
group.  Then $B(G)$ is connected and
\[
 \diam B(G)\leq3.
\]
Consequently, $D(G)$ is connected and
\[
 \diam D(G)\leq3.
\]
\end{theorem}

\begin{proof}
Let $L_1,L_2$ be vertices of $B(G)$.  In particular, neither has prime order.
We show that their distance is at most three.

Suppose first that $L_1$ and $L_2$ are incomparable.  If
$L_1\cap L_2\neq\{e\}$, they are adjacent.  Assume therefore that
$L_1\cap L_2=\{e\}$, and write $\pi(L)$ for the set of prime divisors of
$|L|$.

If neither $\pi(L_1)\subseteq\pi(L_2)$ nor
$\pi(L_2)\subseteq\pi(L_1)$, choose
\[
 p\in\pi(L_1)\setminus\pi(L_2),
 \qquad
 q\in\pi(L_2)\setminus\pi(L_1),
\]
and subgroups $H\leq L_1$, $K\leq L_2$ of orders $p$ and $q$, respectively.
Since $G$ is abelian, $H+K$ is cyclic of order $pq$, and
\[
 L_1\sim H+K\sim L_2.
\]

Now suppose, without loss of generality, that
$\pi(L_1)\subseteq\pi(L_2)$.  Choose \(p\in\pi(L_1)\), and let \(H_1\leq L_1\) and
\(H_2\leq L_2\) be their order-\(p\) subgroups. Since
\(L_1\cap L_2=\{e\}\), we have \(H_1\neq H_2\).

Since $G$
is not a $p$-group, choose an order-$q$ subgroup $K$ with $q\neq p$.  If
$K\nleq L_1,L_2$, then
\[
 L_1\sim H_1+K\sim H_2+K\sim L_2.
\]
If $K\nleq L_1$ but $K\leq L_2$, then
\[
 L_1\sim H_1+K\sim L_2.
\]
The other case is symmetric and gives
$L_1\sim H_2+K\sim L_2$.

It remains to consider comparable vertices, say \(L_1<L_2\).
If \(L_2\) contains every subgroup of prime order in \(G\), then,
for each prime divisor $s$ of $|G|$, the Sylow \(s\)-subgroup of \(G\)
has a unique subgroup of order \(s\). Since \(G\) is abelian, each
Sylow subgroup is therefore cyclic. Since the Sylow subgroups have pairwise coprime orders, their direct product
\(G\) is cyclic which is contrary to the hypothesis.

Thus there exists a subgroup \(K\not\leq L_2\) of prime order, say
\(|K|=p\). Choose a prime-order subgroup \(H\leq L_1\), say
\(|H|=q\). If \(p\neq q\), then \(H+K\) is cyclic and
\[
L_1\sim H+K\sim L_2.
\]

Suppose therefore that \(p=q\). Since \(G\) is not a \(p\)-group,
choose a prime \(r\neq p\) dividing \(|G|\), and let \(Q\) be a
subgroup of order \(r\). If \(Q\not\leq L_2\), then
\[
L_1\sim H+Q\sim L_2.
\]
We may therefore assume that \(Q\leq L_2\). If \(Q\leq L_1\), then
\[
L_1\sim K+Q\sim L_2.
\]
Finally, if \(Q\not\leq L_1\), then
\[
L_1\sim H+Q\sim K+Q\sim L_2.
\]
Thus in every case \(d_{B(G)}(L_1,L_2)\leq3\).

Finally, $B(G)$ is non-empty.  Since $G$ is not cyclic, some Sylow $p$-subgroup
is non-cyclic and contains two distinct order-$p$ subgroups $P_1,P_2$.  Choose
an order-$q$ subgroup $Q$ with $q\neq p$.  Then $P_1+Q$ and $P_2+Q$ are distinct
cyclic subgroups of order $pq$ and are adjacent in $B(G)$.  The assertion for
$D(G)$ follows from Corollary~\ref{cor:D-transfer}.
\end{proof}

\begin{example}[Sharpness]\label{ex:sharp}
Let
\[
 G=\Int_3\times \Int_4\times \Int_4.
\]
Set
\[
 L_1=\langle(0,1,0)\rangle,
 \qquad
 L_2=\langle(0,0,1)\rangle.
\]
The groups $L_1$ and $L_2$ have trivial intersection.  Their unique order-$2$
subgroups are
\[
 H_1=\langle(0,2,0)\rangle,
 \qquad
 H_2=\langle(0,0,2)\rangle.
\]
No cyclic subgroup can be adjacent to both $L_1$ and $L_2$, because it would
have to contain the two distinct order-$2$ subgroups $H_1$ and $H_2$.  Hence
$d_{B(G)}(L_1,L_2)>2$.  On the other hand, with
$K=\langle(1,0,0)\rangle$,
\[
 L_1\sim H_1+K\sim H_2+K\sim L_2.
\]
Thus $\diam B(G)=3$, and consequently $\diam D(G)=3$ (see Figure~\ref{fig:zn344}).
\end{example}

\begin{figure}[ht]
\begin{center}
\includegraphics[scale=0.5]{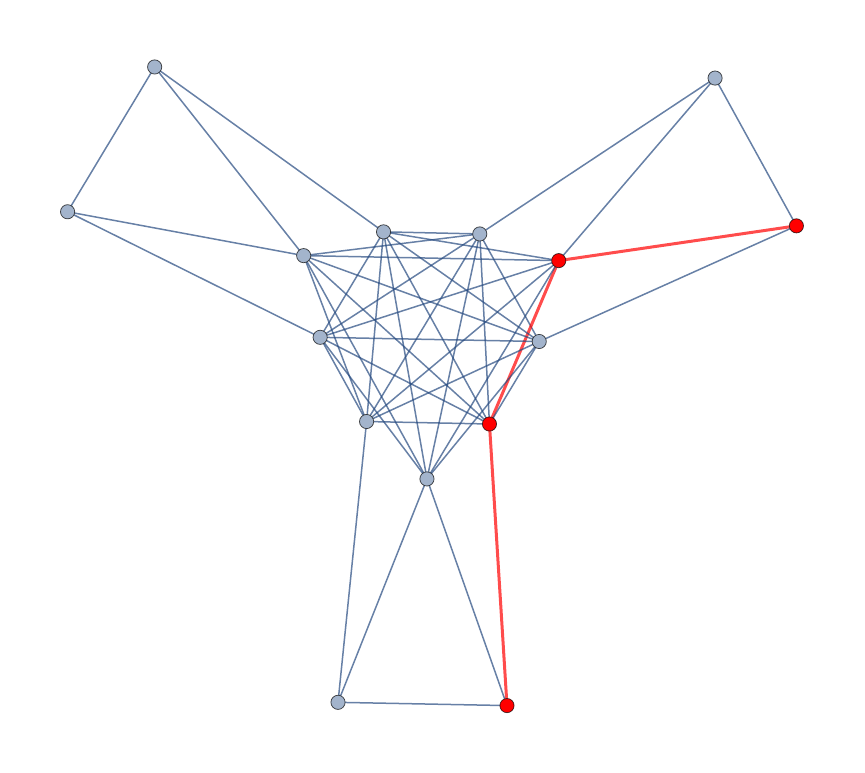} 
\caption{The graph $B(\Int_3\times \Int_4\times \Int_4)$ in Example~\ref{ex:sharp}. A diameter path is highlighted in red.}\label{fig:zn344}
\end{center}
\end{figure}


\section{Non-abelian groups with non-trivial center}

\subsection{Centers which are not \texorpdfstring{$p$}{p}-groups}

Bera and Cameron \cite{BeraCameron2025} proved that if the order of
\(Z(G)\) is divisible by at least two distinct primes, then \(D(G)\)
is connected, apart from the cyclic group of order \(pq\), and in
every such connected case
\[
\operatorname{diam}D(G)\leq 6.
\]
For non-abelian groups, we sharpen this diameter bound to \(4\). 

\begin{lemma}
Let \(G\) be a finite non-abelian group such that \(Z(G)\) is not a
\(p\)-group. Choose subgroups \(H,K\leq Z(G)\) of distinct prime
orders and put \(P=HK\). Then \(P\in V(B(G))\) and
\[
\operatorname{ecc}_{B(G)}(P)\leq2.
\]
\end{lemma}

\begin{proof}
Write \(|H|=p\) and \(|K|=q\), where \(p\neq q\). Since \(H,K\leq
Z(G)\), the subgroup \(P=HK\) is cyclic of order \(pq\). We consider several cases, in each of which we show that $P$ is adjacent to some cyclic subgroup of $G$, which will in particular show that $P\in V(B(G))$.

Suppose \(L\in V(B(G))\) such that \(P\nleq L\). If \(L\)
contains exactly one of \(H,K\), then \(L\sim P\). If it contains
neither, choose a prime-order subgroup \(M\leq L\), and choose
\(C\in\{H,K\}\) such that \(|C|\neq |M|\). Since \(C\leq Z(G)\),
the subgroup \(CM\) is cyclic, and
\[
P\sim CM\sim L.
\]

Now suppose that \(P\leq L\). Assume first that there exists a
maximal cyclic subgroup \(T\) with \(P\nleq T\). If \(T\) contains
exactly one of \(H,K\), then
\[
P\sim T\sim L.
\]
Otherwise \(T\) contains neither \(H\) nor \(K\). Since \(T\nleq L\),
some Sylow \(r\)-subgroup \(R\) of the cyclic group \(T\) is not
contained in \(L\). Choose \(C\in\{H,K\}\) with \(|C|\neq r\).
Then \(CR\) is cyclic and
\[
P\sim CR\sim L.
\]

Finally, suppose that every maximal cyclic subgroup of \(G\) contains
\(P\). Choose a maximal cyclic subgroup \(T_1\) containing \(L\).
Since \(G\) is non-abelian, there exists another maximal cyclic
subgroup \(T_2\neq T_1\). Some Sylow \(r\)-subgroup \(R\) of \(T_2\)
is not contained in \(T_1\). If \(r=p\) or \(r=q\), take \(A=R\)
(with the two cases treated symmetrically); otherwise take \(A=HR\).
Then \(A\) is cyclic, contains exactly one of \(H,K\), and is not
contained in \(T_1\). Consequently
\[
P\sim A\sim L.
\]

It remains only to note that \(B(G)\) is non-empty and that
\(P\in V(B(G))\). If there exists a maximal cyclic subgroup \(T\)
with \(P\nleq T\), then the same argument as above shows that either
\(P\sim T\), or there is a cyclic subgroup \(CR\) such that
\(P\sim CR\). On the other hand, if every maximal cyclic subgroup
contains \(P\), then, since \(G\) is non-abelian, there are two
distinct maximal cyclic subgroups \(T_1,T_2\), and the construction
in the final case above gives a cyclic subgroup \(A\) such that
\(P\sim A\). Thus in all cases \(P\) has a neighbour in
\(\Gamma(G)\). Hence \(P\in V(B(G))\), and \(B(G)\) is non-empty. Moreover, all the paths constructed above lie in
\(B(G)\).

Thus \(d_{B(G)}(P,L)\leq2\) for every \(L\in V(B(G))\), and hence
\[
\operatorname{ecc}_{B(G)}(P)\leq2.
\]

\vspace{-2em}
\end{proof}

\begin{theorem}\label{thm:nonabelian}
Let \(G\) be a finite non-abelian group such that \(Z(G)\) is not a
\(p\)-group. Then \(B(G)\), and hence \(D(G)\), is connected and
non-empty. Moreover,
\[
\operatorname{diam}B(G)\leq4
\qquad\text{and}\qquad
\operatorname{diam}D(G)\leq4.
\]
\end{theorem}

\begin{proof}
By the preceding lemma, $B(G)$ is nonempty and every vertex of \(B(G)\) lies at distance at
most \(2\) from \(HK\). Hence \(B(G)\) is connected, and for any
\(L_1,L_2\in V(B(G))\),
\[
d_{B(G)}(L_1,L_2)
 \leq d_{B(G)}(L_1,HK)+d_{B(G)}(HK,L_2)
 \leq4.
\]
Thus \(\operatorname{diam}B(G)\leq4\). The assertions for \(D(G)\)
follow from Corollary \ref{cor:D-transfer}.
\end{proof}

The following example shows that the bound in Theorem \ref{thm:nonabelian} is sharp.

\begin{example}[Sharpness]\label{ex:sharpc25c4}
Let
\[
G=(C_{25}\rtimes C_4)\times C_{10},
\]
where
\[
C_{25}=\langle a\rangle,\qquad
C_4=\langle b\rangle,\qquad
C_{10}=\langle c\rangle,
\]
and the action of \(b\) on \(C_{25}\) is given by
\[
bab^{-1}=a^7.
\]
Since $7^2\equiv -1\pmod{25}$, the element $7$ has order $4$
modulo $25$.

We identify these generators with their natural images in $G$; for example,
we write $bc$ and $ac$ in place of $(b,c)$ and $(a,c)$, respectively.

We first note that
\[
Z(C_{25}\rtimes C_4)=\{e\}.
\]
Indeed, if \(a^ib^j\) commutes with \(a\), then
\(7^j\equiv1\pmod{25}\), and hence \(j\equiv0\pmod4\).
Thus the element is \(a^i\). Commuting also with \(b\) gives
\[
a^{7i}=a^i,
\]
so \(6i\equiv0\pmod{25}\), whence \(i\equiv0\pmod{25}\).
Consequently
\[
Z(G)=\langle c\rangle\cong C_{10},
\]
which is not a \(p\)-group.

Put
\[
L_1=\langle b\rangle,\qquad L_2=\langle a\rangle,
\qquad P=\langle c\rangle.
\]
There is a path
\[
\langle b\rangle
 \sim \langle bc\rangle
 \sim \langle c\rangle
 \sim \langle ac\rangle
 \sim \langle a\rangle.
\]
Indeed,
\[
\langle b\rangle\cap\langle bc\rangle
   =\langle b^2\rangle,
\qquad
\langle bc\rangle\cap\langle c\rangle
   =\langle c^2\rangle,
\]
while
\[
\langle c\rangle\cap\langle ac\rangle
   =\langle c^5\rangle,
\qquad
\langle ac\rangle\cap\langle a\rangle
   =\langle a^5\rangle,
\]
and in each case the two cyclic subgroups are incomparable.
Hence
\[
d_{B(G)}(L_1,L_2)\leq4.
\]

We next show that \(d_{B(G)}(L_1,L_2)>2\). Clearly
\[
L_1\cap L_2=\{e\},
\]
so \(L_1\nsim L_2\). Suppose that \(X\) were a common neighbour of
\(L_1\) and \(L_2\). Since \(X\sim\langle b\rangle\), we have
\[
\langle b^2\rangle\leq X,
\]
while \(X\sim\langle a\rangle\) gives
\[
\langle a^5\rangle\leq X.
\]
Since \(X\) is cyclic, the elements \(b^2\) and \(a^5\) must commute.
However,
\[
b^2a^5b^{-2}
 =a^{5\cdot 7^2}
 =a^{245}
 =a^{20}
 =a^{-5}\neq a^5.
\]
This contradiction shows that \(L_1\) and \(L_2\) have no common
neighbour. Hence
\[
d_{B(G)}(L_1,L_2)>2.
\]

It remains to exclude a path of length three. Suppose
\[
L_1\sim X\sim Y\sim L_2.
\]
Since \(X\sim\langle b\rangle\), the subgroup \(X\) contains
\(\langle b^2\rangle\). Hence
\[
X\leq C_G(b^2)=\langle b\rangle\times\langle c\rangle.
\]
Similarly, since \(Y\sim\langle a\rangle\), the subgroup \(Y\)
contains \(\langle a^5\rangle\), and therefore
\[
Y\leq C_G(a^5)=\langle a\rangle\times\langle c\rangle.
\]

\begin{figure}[t]
\begin{center}
\includegraphics[scale=0.85]{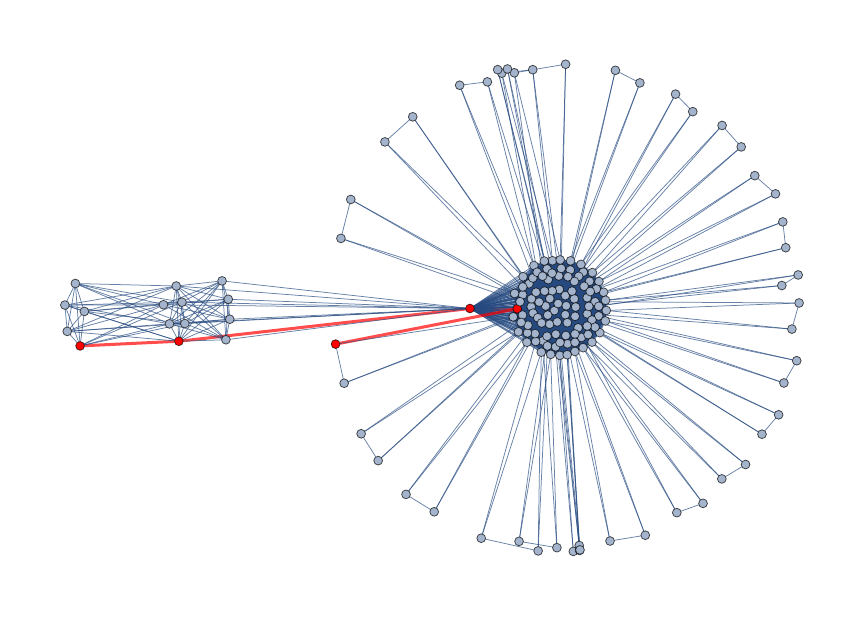} 
\caption{The graph $B(G)$ in Example~\ref{ex:sharpc25c4}. A diameter path is highlighted in red.}\label{fig:c25c4}
\end{center}
\end{figure}

Consequently,
\[
X\cap Y\leq\langle c\rangle.
\]

Now \(X\) is cyclic and contains the order-\(2\) subgroup
\(\langle b^2\rangle\). Hence \(X\cap\langle c\rangle\) cannot
contain the distinct order-\(2\) subgroup \(\langle c^5\rangle\),
since a cyclic group has a unique subgroup of order \(2\).
Thus
\[
|X\cap\langle c\rangle|
\]
is odd.

On the other hand, \(Y\) is cyclic and contains the order-\(5\)
subgroup \(\langle a^5\rangle\). Hence
\(Y\cap\langle c\rangle\) cannot contain the distinct order-\(5\)
subgroup \(\langle c^2\rangle\). Thus
\[
|Y\cap\langle c\rangle|
\]
is not divisible by \(5\).

Since \(|\langle c\rangle|=10\), it follows that
\[
X\cap Y=\{e\},
\]
contrary to \(X\sim Y\). Therefore
\[
d_{B(G)}(\langle b\rangle,\langle a\rangle)=4.
\]
Hence
\[
\operatorname{diam}B(G)=4.
\]
By Corollary \ref{cor:D-transfer},
\[
\operatorname{diam}D(G)=4.
\]
Thus the bound in Theorem \ref{thm:nonabelian} is sharp (see Figure~\ref{fig:c25c4}).
\end{example}

\subsection{Non-trivial \texorpdfstring{$p$}{p}-group centers}

We now consider finite groups \(G\) which are not \(p\)-groups and for
which
\[
\set{e}\neq Z(G)\quad \text{is a \(p\)-group.}
\]
The rooted-tree description used for \(p\)-groups no
longer applies to arbitrary cyclic overgroups of a prime-order subgroup in this case.
We therefore begin by extending the notions of branching and
\(b\)-normality.

\begin{definition}
Let \(G\) be a finite group and let \(P\leq G\) have prime order. We
call \(P\) \emph{branching} if there exist incomparable cyclic subgroups
\(H,K\leq G\) such that
\[
P\leq H\cap K.
\]
\end{definition}

For \(p\)-groups this agrees with the branching notion introduced in
Section~4.

\begin{definition}
Let \(L\) be a finite group and let \(P\trianglelefteq L\) have prime
order. We call \(P\) \emph{\(b\)-normal in \(L\)} if at least one of
the following holds:
\begin{enumerate}
    \item[\textnormal{(B1)}] \(P<H\) for some cyclic subgroup
    \(H\leq L\) which is not normal in \(L\);
    \item[\textnormal{(B2)}] there exist incomparable normal cyclic
    subgroups \(H,K\trianglelefteq L\) such that
    \[
    P\leq H\cap K.
    \]
\end{enumerate}
\end{definition}

For \(p\)-groups this definition is equivalent to Definition~\ref{def:bnormal}. Indeed, if \(H\) and \(K\) in \textnormal{(B2)} have different
orders, say \(|H|<|K|\), then the unique subgroup \(K_0\) of \(K\)
having order \(|H|\) is characteristic in \(K\), hence normal in
\(L\); moreover \(H\neq K_0\), and \(H,K_0\) give the condition
\textnormal{(B2)} of Definition~\ref{def:bnormal}.

\begin{proposition}\label{prop:branching}
Let \(G\) be a finite group and let \(P\leq G\) have prime order.
Then \(P\) is branching in \(G\) if and only if \(P\) is \(b\)-normal
in \(N_G(P)\).
\end{proposition}

\begin{proof}
Every cyclic subgroup containing \(P\) centralizes \(P\), and hence is
contained in \(C_G(P)\leq N_G(P)\).

Suppose first that \(P\) is branching, witnessed by incomparable cyclic
subgroups \(H,K\). If one of \(H,K\) is not normal in \(N_G(P)\), then
\textnormal{(B1)} holds. If both are normal, then \textnormal{(B2)}
holds.

Conversely, \textnormal{(B2)} immediately implies that \(P\) is
branching. If \textnormal{(B1)} holds, choose
\(g\in N_G(P)\) such that \(H^g=gHg^{-1}\neq H\). Since \(P\trianglelefteq
N_G(P)\),
\[
P\leq H\cap H^g.
\]
The cyclic subgroups \(H\) and \(H^g\) have the same order and are
distinct, hence are incomparable. Thus \(P\) is branching.
\end{proof}

\begin{corollary}[Non-emptiness criterion]\label{cor:nonempty1}
For every finite group \(G\),
\(
B(G)\neq\varnothing 
\) (equivalently, \(
D(G)\neq\varnothing
\))
if and only if there exists a prime-order subgroup \(P\leq G\) which
is \(b\)-normal in \(N_G(P)\). 
\end{corollary}

\begin{proof}
If \(P\) is branching, its two incomparable cyclic overgroups form an
edge of \(B(G)\). Conversely, if \(H\sim K\), choose a prime-order
subgroup
\[
P\leq H\cap K.
\]
Then \(P\) is branching. The assertion for \(D(G)\) follows from
Proposition \ref{prop:null-composition}.
\end{proof}

\begin{definition}[Central hub]
Let $G$ be a finite group which is not a $p$-group and such that
$\set{e}\neq Z(G)$ is a $p$-group. Then $G$ is necessarily non-abelian. Fix a subgroup
\[
P\leq Z(G),\qquad |P|=p,
\]
and put
\[
\mu(G)=\pi(G)\setminus\{p\}.
\]
Define
\[
\mathcal H_P
=
\{PQ: Q\leq G,\ |Q|=q\text{ for some }q\in\mu(G)\}.
\]
We call $\mathcal H_P$ the central hub associated with $P$.
\end{definition}

For every $PQ\in\mathcal H_P$, the subgroup $PQ$ is cyclic of order
$pq$. Moreover, distinct members of $\mathcal H_P$ are adjacent,
since they contain $P$ and are incomparable. Thus, whenever
$|\mathcal H_P|\geq2$, the set $\mathcal H_P$ induces a clique in
$B(G)$.

\subsubsection{The single-hub case}

Suppose that
\[
|\mathcal H_P|=1.
\]
Then
\[
\pi(G)=\{p,q\}
\]
for some prime \(q\neq p\), and \(G\) has a unique subgroup \(Q\) of
order \(q\). Thus
\[
|G|=p^i q^j.
\]

\begin{lemma}
The subgroup \(Q\) is not branching if and only if
\[
j=1
\qquad\text{and}\qquad
C_G(Q)\text{ is cyclic}.
\]
\end{lemma}

\begin{proof}
Suppose first that \(Q\) is not branching. We show that \(j=1\).
If \(j>1\), let \(S\) be a Sylow \(q\)-subgroup of \(G\). Since \(Q\)
is the unique subgroup of order \(q\), the \(q\)-group \(S\) has a
cyclic subgroup \(K\) such that
\[
Q<K.
\]
Choose \(P\leq Z(G)\) of order \(p\). Since \(P\) centralizes \(Q\),
the subgroup \(PQ\) is cyclic of order \(pq\). Moreover,
\[
Q\leq K\cap PQ,
\]
while \(K\) and \(PQ\) are incomparable. Hence
\[
K\sim PQ,
\]
so \(Q\) is branching, a contradiction. Therefore \(j=1\).

Put
\[
C=C_G(Q).
\]
Suppose that \(A\) and \(B\) are incomparable cyclic \(p\)-subgroups
of \(C\). Since both centralize \(Q\), the subgroups \(AQ\) and \(BQ\)
are cyclic. They both contain \(Q\), and they are incomparable.
Therefore
\[
AQ\sim BQ,
\]
again contradicting the assumption that \(Q\) is not branching.
Thus all cyclic \(p\)-subgroups of \(C\) are linearly ordered by
inclusion.

Since every \(p\)-element of \(C\) generates a cyclic \(p\)-subgroup,
there is a largest cyclic \(p\)-subgroup \(A\) of \(C\), and every
\(p\)-element of \(C\) belongs to \(A\). Hence \(A\) is the unique
Sylow \(p\)-subgroup of \(C\). Since \(j=1\) and \(Q\leq Z(C)\), we
have
\[
C=A\times Q.
\]
As \(A\) and \(Q\) are cyclic of coprime orders, \(C\) is cyclic.

Conversely, suppose that \(j=1\) and \(C_G(Q)\) is cyclic. Every
cyclic subgroup of \(G\) containing \(Q\) is contained in \(C_G(Q)\).
Since the \(p\)-part of \(C_G(Q)\) is cyclic, its cyclic overgroups of
\(Q\) are linearly ordered by inclusion. Thus no two cyclic overgroups
of \(Q\) are incomparable, and hence \(Q\) is not branching.
\end{proof}

\begin{proposition}
In the single-hub case,
\(
B(G)=\varnothing
\) 
if and only if
\( 
j=1\text{ and } C_G(Q)\text{ is cyclic},
\) 
and no subgroup of order \(p\) is branching (equivalently, no subgroup \(M\leq G\) of order \(p\) is \(b\)-normal
in \(N_G(M)\)). Moreover,
\(
B(G)=\varnothing
\quad\Longrightarrow\quad
Z(G)\cong C_p.
\)
\end{proposition}

\begin{proof}
By the general non-emptiness criterion, \(B(G)\) is non-empty if and
only if \(G\) has a branching subgroup of prime order. Since
\[
\pi(G)=\{p,q\},
\]
the graph \(B(G)\) is empty if and only if neither the unique subgroup
\(Q\) of order \(q\) nor any subgroup of order \(p\) is branching.
The first condition is, by the preceding lemma, equivalent to
\[
j=1
\qquad\text{and}\qquad
C_G(Q)\text{ is cyclic}.
\]
This proves the first assertion. The equivalent formulation follows
from the characterization of branching in terms of \(b\)-normality.

Finally, suppose that \(B(G)=\varnothing\). Since \(C_G(Q)\) is cyclic
and
\[
Z(G)\leq C_G(Q),
\]
the center \(Z(G)\) is cyclic. If \(|Z(G)|>p\), choose a cyclic subgroup
\(K\leq Z(G)\) of order \(p^2\), and let \(P_0\) be its subgroup of
order \(p\). Since \(Q\) centralizes \(P_0\), the subgroup \(P_0Q\) is
cyclic, and
\[
K\sim P_0Q.
\]
Thus \(P_0\) is branching, a contradiction. Hence
\[
|Z(G)|=p,
\]
and therefore \(Z(G)\cong C_p\).
\end{proof}

\begin{theorem}[The single-hub case]\label{thm:singhub}
Let $G$ be a finite (non-abelian) group which is not a $p$-group and such that
$\set{e}\neq Z(G)$ is a $p$-group. Fix a subgroup $P\leq Z(G)$ of order
$p$, and suppose that $|\mathcal H_P|=1$. 
Assume that $B(G)$ is non-empty, and let $Q$ be the unique subgroup
of $G$ of order $q$. Then \(B(G)\) is connected if and
only if
\[
C_G(M)\text{ is not a \(p\)-group}
\]
for every non-central branching subgroup \(M\leq G\), $(M\nleq Z(G))$, of order \(p\).

Equivalently, the same condition holds for every non-central subgroup
\(M\) of order \(p\) which is \(b\)-normal in \(N_G(M)\).

Whenever these conditions hold,
\[
PQ\in V(B(G)),
\qquad
\operatorname{ecc}_{B(G)}(PQ)\leq2.
\]
Consequently,
\[
\operatorname{diam}B(G)\leq4,
\qquad
\operatorname{diam}D(G)\leq4.
\]
\end{theorem}

\begin{proof}
Suppose first that \(M\) is a non-central branching subgroup of order
\(p\) such that \(C_G(M)\) is a \(p\)-group. Since \(M\) is branching,
there are incomparable cyclic subgroups \(A\) and \(A'\) containing
\(M\). Every cyclic subgroup containing \(M\) lies in \(C_G(M)\), and
hence \(A\) and \(A'\) are cyclic \(p\)-groups.

We claim that the whole component containing \(A\) consists of cyclic
\(p\)-groups containing \(M\). Indeed, let \(K\) be a cyclic \(p\)-group
containing \(M\), and suppose that \(L\sim K\). Since
\(K\cap L\neq\{e\}\), and every non-trivial subgroup of the cyclic
\(p\)-group \(K\) contains its unique subgroup \(M\) of order \(p\), we
have
\[
M\leq K\cap L.
\]
Thus \(L\leq C_G(M)\), and therefore \(L\) is again a cyclic \(p\)-group
containing \(M\). The claim follows inductively.

Now \(Q\) is normal in \(G\), since it is the unique subgroup of order
\(q\). Also \(Q\nleq C_G(M)\), because \(C_G(M)\) is a \(p\)-group.
Hence \(M\) cannot be normal in \(G\): otherwise the two normal
subgroups \(M\) and \(Q\), having coprime orders, would centralize one
another. Choose \(g\in G\) such that \(M^g\neq M\). Then \(M^g\) is
again branching and
\[
C_G(M^g)=C_G(M)^g
\]
is a \(p\)-group. Thus there is another component consisting of cyclic
\(p\)-groups containing \(M^g\). These two components are distinct,
since a cyclic \(p\)-group cannot contain two distinct subgroups of
order \(p\). Therefore \(B(G)\) is disconnected.

Conversely, suppose that
\[
C_G(M)\ \text{is not a \(p\)-group}
\]
for every non-central branching subgroup \(M\) of order \(p\). We show
that every vertex of \(B(G)\) has distance at most \(2\) from \(PQ\).

Let \(L\in V(B(G))\), and choose \(T\sim L\). Choose a subgroup
\[
R\leq L\cap T
\]
of prime order. Then \(R\) is branching.

Suppose first that \(|R|=q\). Since \(Q\) is the unique subgroup of
order \(q\), we have \(R=Q\). If \(P\nleq L\), then
\[
L\sim PQ.
\]
If \(P\leq L\), then \(PQ\leq L\). If \(P\nleq T\), then
\[
L\sim T\sim PQ.
\]
If \(P\leq T\), choose a Sylow subgroup \(S\) of the cyclic group \(T\)
which is not contained in \(L\). Since \(P,Q\leq L\cap T\), we have
\(S>P\) or \(S>Q\), according as \(S\) is a \(p\)-group or a
\(q\)-group (since $\pi(S)\subseteq\pi(G)=\set{p,q}$). In either case
\[
L\sim S\sim PQ.
\]

Now suppose that \(|R|=p\). If \(R=P\), the same
argument gives
\[
d_{B(G)}(L,PQ)\leq 2.
\]

Finally, suppose that \(R\neq P\). Since \(L\) is cyclic, \(P\nleq L\).
If \(Q\leq L\), then
\[
L\sim PQ.
\]
Suppose therefore that \(Q\nleq L\). If \(R\) is central, then
\(C_G(R)=G\); if \(R\) is non-central, our hypothesis gives that
\(C_G(R)\) is not a \(p\)-group. Since the only prime divisors of
\(|G|\) are \(p\) and \(q\), and \(Q\) is the unique subgroup of order
\(q\), it follows that
\[
Q\leq C_G(R).
\]
Hence \(RQ\) is cyclic of order \(pq\), and
\[
L\sim RQ\sim PQ.
\]

Thus every vertex of \(B(G)\) is at distance at most \(2\) from \(PQ\).
Since \(B(G)\) is non-empty, this also shows that \(PQ\in V(B(G))\).
Hence \(B(G)\) is connected and
\[
\operatorname{ecc}_{B(G)}(PQ)\leq 2.
\]
Therefore
\[
\operatorname{diam} B(G)\leq 4.
\]
The corresponding assertion for \(D(G)\) follows from Corollary \ref{cor:D-transfer}.
\end{proof}

\begin{example}
Let
\[
G=C_2\times S_3.
\]
Then \(Z(G)\cong C_2\), and \(G\) has a unique subgroup \(Q\) of order
\(3\). Moreover,
\[
C_G(Q)\cong C_6,
\]
and no subgroup of order \(2\) is branching. Hence
\[
B(G)=\varnothing.
\]
\end{example}

\subsubsection{The multiple-hub case}

Suppose now that
\[
|\mathcal H_P|\geq2.
\]
Then \(B(G)\) is automatically non-empty as we shall see.

\begin{theorem}[The multiple-hub case]\label{thm:mulhub}
Let $G$ be a finite (non-abelian) group which is not a $p$-group and such that
$\set{e}\neq Z(G)$ is a $p$-group. Fix a subgroup $P\leq Z(G)$ of order
$p$, and suppose that $|\mathcal H_P|\geq 2$. 
Then \(B(G)\) is connected if and only if
\[
C_G(M)\text{ is not a \(p\)-group}
\]
for every non-central branching subgroup \(M\leq G\) of order \(p\).

Equivalently, the same condition holds for every non-central subgroup
\(M\) of order \(p\) which is \(b\)-normal in \(N_G(M)\).

Whenever these equivalent conditions hold, every vertex of \(B(G)\)
has distance at most \(2\) from the clique \(\mathcal H_P\).
Consequently,
\[
\operatorname{diam}B(G)\leq5,
\qquad
\operatorname{diam}D(G)\leq5.
\]
\end{theorem}

\begin{proof}
First observe that \(\mathcal H_P\) induces a clique in \(B(G)\).
Indeed, let \(PQ_1\) and \(PQ_2\) be two distinct members of
\(\mathcal H_P\). They have the non-trivial common subgroup \(P\). If
\(|Q_1|\neq |Q_2|\), neither of \(PQ_1\) and \(PQ_2\) contains the
other. If \(|Q_1|=|Q_2|\), they are distinct cyclic subgroups of the
same order, and hence are again incomparable. Thus
\[
PQ_1\sim PQ_2.
\]
Since \(|\mathcal H_P|\geq2\), every member of \(\mathcal H_P\) is a
vertex of \(B(G)\). 

\begin{center}
Let \(\mathcal C\) denote the component containing the clique \(\mathcal H_P\).
\end{center}

Suppose that \(M\) is a non-central branching subgroup of order \(p\)
and that \(C_G(M)\) is a \(p\)-group. Since \(M\) is branching, some
cyclic \(p\)-group \(K\) containing \(M\) is a vertex of \(B(G)\).
As in the proof of the preceding theorem, every vertex in the
component containing \(K\) is a cyclic \(p\)-group containing \(M\).
No member of \(\mathcal H_P\), whose order is divisible by a prime
different from \(p\), can belong to this component. Hence \(B(G)\) is
disconnected. This proves the necessity.

Conversely, assume that \(C_G(M)\) is not a \(p\)-group for every
non-central branching subgroup \(M\) of order \(p\). We show that every
vertex of \(B(G)\) lies in \(\mathcal C\).

We first consider primary cyclic vertices, that is, cyclic subgroups of $G$ whose orders are powers of a prime. Let \(K\in V(B(G))\) be a
cyclic \(r\)-group, and let \(R\) be its unique subgroup of order \(r\).
Since a prime-order subgroup is isolated in \(\Gamma(G)\), we have
\(K>R\).

If \(r\neq p\), then \(PR\in\mathcal H_P\), and
\[
K\sim PR.
\]
Hence \(K\in\mathcal C\).

Now let \(r=p\), and write \(R=M\). Since \(K\) is a vertex, choose
\(T\sim K\). Then
\[
M\leq K\cap T,
\]
so \(M\) is branching. If \(M=P\), then for every
\(PQ\in\mathcal H_P\),
\[
K\sim PQ.
\]
Thus \(K\in\mathcal C\).

Suppose that \(M\neq P\). If \(M\leq Z(G)\), then \(C_G(M)=G\), while
if \(M\) is non-central, \(C_G(M)\) is not a \(p\)-group by hypothesis.
Thus in either case there is a prime \(q\neq p\) and a subgroup
\[
Q\leq C_G(M),\qquad |Q|=q.
\]
Then \(MQ\) is cyclic of order \(pq\), and
\[
K\sim MQ\sim PQ.
\]
Hence every primary vertex belongs to \(\mathcal C\).

Next, let
\[
U=Q_1Q_2
\]
be cyclic, where \(Q_1\) and \(Q_2\) have distinct prime orders. We
claim that \(U\in\mathcal C\). If one of \(Q_1,Q_2\) is \(P\), then
\(U\in\mathcal H_P\). If \(p\nmid |U|\), say \(Q_1\) has order
\(q_1\), then
\[
U\sim PQ_1,
\]
so \(U\in\mathcal C\). Finally, if \(p\mid |U|\), let \(M\neq P\) be
the order-\(p\) subgroup of \(U\), and let \(Q\) be its other
prime-order subgroup. Then
\[
U=MQ\sim PQ.
\]
Thus every cyclic subgroup whose order is the product of two distinct
primes belongs to \(\mathcal C\).

It remains to consider a vertex \(L\in V(B(G))\) whose order is
divisible by at least two distinct primes. Choose \(T\sim L\). Since
\(T\nleq L\), some Sylow subgroup \(S\) of the cyclic group \(T\) is
not contained in \(L\).

If
\[
S\cap L\neq\{e\},
\]
then \(L\nleq S\), since \(L\) has mixed order, and therefore
\[
L\sim S.
\]
Let \(R\) be the unique subgroup of prime order contained in \(S\).

If \(|R|=r\neq p\), then
\[
S\sim PR\in\mathcal H_P,
\]
and hence \(d_{B(G)}(L,\mathcal H_P)\leq2\).

Suppose now that \(|R|=p\), and write \(R=M\). If \(M=P\), then
\[
S\sim PQ
\]
for every \(PQ\in\mathcal H_P\), and again
\[
d_{B(G)}(L,\mathcal H_P)\leq2.
\]

Finally, suppose that \(M\neq P\). Since
\[
M\leq S\cap L,
\]
the subgroup \(M\) is branching. If \(M\leq Z(G)\), then
\(C_G(M)=G\); if \(M\) is non-central, the hypothesis implies that
\(C_G(M)\) is not a \(p\)-group. Thus there exist a prime
\(q\neq p\) and a subgroup
\[
Q\leq C_G(M),\qquad |Q|=q.
\]
Then \(MQ\) is cyclic.

If \(Q\nleq L\), then
\[
L\sim MQ\sim PQ.
\]
If \(Q\leq L\), then \(P\nleq L\), since the cyclic group \(L\)
cannot contain the two distinct subgroups \(M\) and \(P\) of order
\(p\). Hence
\[
L\sim PQ.
\]
Thus in every case, here, 
\[
d_{B(G)}(L,\mathcal H_P)\leq2.
\]
and so $L\in\mathcal{C}$.


Suppose now that
\[
S\cap L=\{e\}.
\]
Since \(L\sim T\), choose a subgroup
\[
Q_1\leq L\cap T
\]
of prime order. Let \(Q_2\) be the unique subgroup of prime order
contained in \(S\). Then \(Q_2\nleq L\). Moreover,
\[
|Q_1|\neq |Q_2|,
\]
for otherwise \(Q_1=Q_2\), since \(T\) is cyclic, contradicting
\(S\cap L=\{e\}\).

Hence
\[
U=Q_1Q_2
\]
is cyclic of order the product of two distinct primes. We have
\[
Q_1\leq L\cap U
\]
and \(U\nleq L\), since \(Q_2\nleq L\). Also \(L\nleq U\), since
otherwise \(L\leq U\leq T\), contrary to \(L\sim T\). Therefore
\[
L\sim U.
\]
By the preceding paragraph, \(U\) is either a member of
\(\mathcal H_P\) or is adjacent to a member of \(\mathcal H_P\).
Hence
\[
d_{B(G)}(L,\mathcal H_P)\leq2.
\]
and consequently \(L\in\mathcal C\).

Thus every vertex of \(B(G)\) lies in \(\mathcal C\), and hence
\(B(G)\) is connected. Moreover, as shown above, every vertex has
distance at most \(2\) from some vertex of \(\mathcal H_P\). Since \(\mathcal H_P\) is a clique with at least two vertices, we have
\[
\operatorname{diam}B(G)\leq5,
\]
and consequently
\[
\operatorname{diam}D(G)\leq5
\]
by Corollary~\ref{cor:D-transfer}.
\end{proof}

\begin{example}
Let
\[
G=D_{18}\times C_3,
\]
where
\[
D_{18}=\langle r,s:r^9=s^2=e,\ srs=r^{-1}\rangle,
\]
and let $P$ be the central subgroup of order $3$ coming from the
second factor. Then $Z(G)=P\cong C_3$. 
The nine reflection subgroups
\[
\langle r^i s\rangle,\qquad i=0,1,\ldots,8,
\]
of $D_{18}$ are precisely its subgroups of order $2$, and hence they give the nine members of $\mathcal H_P$. So 
\[
|\mathcal H_P|=9.
\]
The subgroup
\[
M=\langle r^3\rangle
\]
is a non-central branching subgroup of order \(3\), while
\[
C_G(M)=\langle r\rangle\times C_3
\]
is a \(3\)-group. Hence \(B(G)\) is disconnected. In fact,
\[
B(G)\cong K_9\dot\cup K_3.
\]
\end{example}

\begin{example}
Let
\[
G=D_{10}\times C_5,
\]
and let $P$ be the central subgroup of order $5$ coming from the
second factor. Then $Z(G)=P\cong C_5$ and
\[
|\mathcal H_P|=5.
\]
There is no non-central branching subgroup of order $5$, so the
centralizer condition is vacuous. Hence $B(G)$ is connected; in fact,
\[
B(G)\cong K_5.
\]
\end{example}

\begin{example}[Sharpness of the bound $5$]\label{ex:fxc2}
Let
\[
G=F\times C_2,
\qquad
F=C_{13}\rtimes C_{12}
 =\langle a,b\mid a^{13}=b^{12}=e,\;bab^{-1}=a^2\rangle,
\]
and write \(C_2=\langle c\rangle\). 
We first note that
\[
Z(F)=\{e\}.
\]
Indeed, if \(a^ib^j\) commutes with \(a\), then
\[
2^j\equiv1\pmod {13},
\]
and hence \(j\equiv0\pmod {12}\). Commuting also with \(b\) gives
\(2i\equiv i\pmod {13}\), and therefore \(i=0\). Consequently
\[
Z(G)=\langle c\rangle\cong C_2.
\]

We next verify that $B(G)$ is connected. Let
\[
\pi_F:G=F\times C_2\longrightarrow F
\]
denote the natural {\em projection} onto the first factor. Recall that
$\langle b\rangle\cong C_{12}$ is a {\em complement} to the normal subgroup
$\langle a\rangle\cong C_{13}$ in $F$, that is,
\[
F=\langle a\rangle\langle b\rangle
\qquad\text{and}\qquad
\langle a\rangle\cap\langle b\rangle=\{e\}.
\]
Every non-central involution of $G$ has non-trivial projection onto $F$.
Every involution of $F$ lies in a conjugate of the complement
$\langle b\rangle$, and the centralizer in $F$ of any non-trivial
element of such a complement is the complement itself. Hence, for every
non-central subgroup $M\leq G$ of order $2$, the group $C_G(M)$
contains a cyclic subgroup of order $12$, and in particular is not a
$2$-group. Thus the condition of Theorem~\ref{thm:mulhub} is satisfied, so $B(G)$
is connected and
\[
\diam B(G)\leq 5.
\]

We show that equality holds. Put
\[
h=a^2b.
\]
Since
\[
h=a^{11}ba^{-11},
\]
the subgroup \(\langle h\rangle\) is a conjugate of
\(\langle b\rangle\), and hence \(h\) has order \(12\). Moreover,
\(\langle h\rangle\neq\langle b\rangle\).

Consider the cyclic subgroups
\[
L_1=\langle b^3c\rangle,
\qquad
L_2=\langle h^3\rangle,
\]
both of order \(4\). There is a path
\[
L_1
\sim \langle b\rangle
\sim \langle b^4c\rangle
\sim \langle h^4c\rangle  
\sim \langle hc\rangle
\sim L_2 .
\]
Indeed, $|b|=|h|=12$ and $|c|=2$. The successive intersections are
\[
\langle b^6\rangle,\quad
\langle b^4\rangle,\quad
\langle c\rangle,\quad
\langle h^4\rangle,\quad
\langle h^6\rangle,
\]
which have orders
\[
2,3,2,3,2,
\]
respectively; in each case the two cyclic subgroups are incomparable.
Hence
\[
d_{B(G)}(L_1,L_2)\leq5.
\]

It remains to show that no shorter path is possible. The unique
subgroup of order \(2\) in \(L_1\) is
\[
\langle b^6\rangle.
\]
Thus every neighbour of \(L_1\) contains \(b^6\), and hence lies in
\[
C_G(b^6)=\langle b\rangle\times\langle c\rangle.
\]
A direct inspection of the cyclic subgroups of
$C_{12}\times C_2$ shows that the neighbours of $L_1$ in $B(G)$
are precisely
\[
\langle b\rangle,\qquad
\langle b^2\rangle,\qquad
\langle b^3\rangle.
\]
Similarly, the unique subgroup of order $2$ in $L_2$ is
$\langle h^6\rangle$, and
\[
C_G(h^6)=\langle h\rangle\times\langle c\rangle.
\]
Consequently, the neighbours of $L_2$ in $B(G)$ are precisely
\[
\langle h^3c\rangle,\qquad
\langle hc\rangle,\qquad
\langle h^2\rangle.
\]

The two subgroups $\langle b\rangle$ and $\langle h\rangle$ are
distinct complements of $C_{13}$ in $F$, and
therefore
\[
\langle b\rangle\cap\langle h\rangle=\{e\}.
\]
Furthermore, the centralizer in $F$ of every non-identity element
of either complement is that complement itself. Notice also that
none of the neighbours of $L_1$ or $L_2$ contains the central
involution $c$. Hence every non-identity element of a neighbour of
$L_1$ has non-trivial projection in $\langle b\rangle$, and
similarly every non-identity element of a neighbour of $L_2$ has
non-trivial projection in $\langle h\rangle$.

It follows first that no neighbour of $L_1$ is adjacent to a
neighbour of $L_2$. Moreover, no such pair has a common neighbour.
Indeed, suppose that $X$ is a neighbour of $L_1$ and $Y$ is a
neighbour of $L_2$, and that they have a common neighbour $T$.
Choose
\[
e\neq x\in X\cap T,\qquad e\neq y\in Y\cap T.
\]
Since $T$ is cyclic, $x$ and $y$ commute. Their projections to $F$ are
non-trivial commuting elements lying respectively in $\langle b\rangle$
and $\langle h\rangle$. But the centralizer in $F$ of the first
projection is $\langle b\rangle$, forcing the second projection to lie
in
\[
\langle b\rangle\cap\langle h\rangle=\{e\},
\]
a contradiction.

Therefore a path from \(L_1\) to \(L_2\) cannot have length at most
\(4\). Hence
\[
d_{B(G)}(L_1,L_2)=5.
\]
Thus
\[
\boxed{\operatorname{diam}B(G)=5}.
\]
By Corollary~\ref{cor:D-transfer},
\[
\boxed{\operatorname{diam}D(G)=5}.
\]
Therefore the diameter bound in Theorem~\ref{thm:mulhub} is sharp (see Figure~\ref{fig:fc1213c2}).
\end{example}

\begin{figure}[ht]
\begin{center}
\includegraphics[scale=0.6]{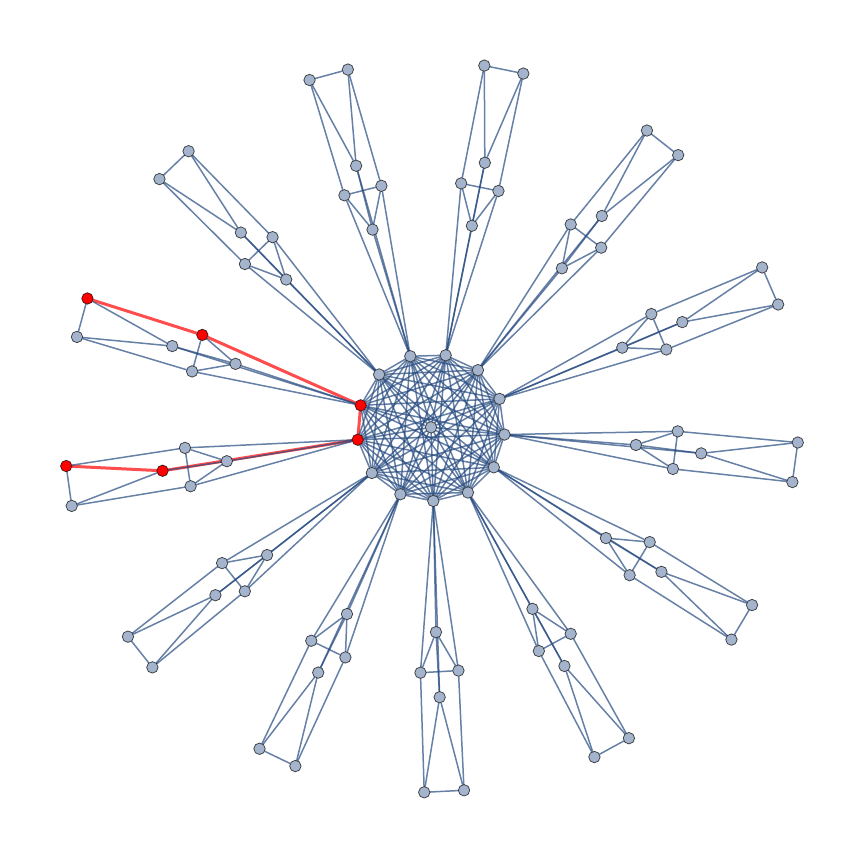} 
\caption{The graph $B(G)$ in Example~\ref{ex:fxc2}. A diameter path is highlighted in red.}\label{fig:fc1213c2}
\end{center}
\end{figure}

Although the present section concerns non-abelian groups, the following
result holds for arbitrary finite groups, and hence includes the abelian
case. Groups with a unique involution were singled out for further study by
Bera and Cameron~\cite{BeraCameron2025}; the following result gives, in particular,
a connectivity and diameter statement for this class.

\begin{theorem}[Groups with a unique involution]
Let $G$ be a finite group with a unique involution. If $B(G)$ is
non-empty, then $B(G)$ is connected and
\[
\diam B(G)\leq 3.
\]
Consequently,
\[
\diam D(G)\leq 3.
\]
\end{theorem}

\begin{proof}
Let $P$ be the unique subgroup of order $2$. Since the involution of
$G$ is unique, $P\leq Z(G)$.

Let $H,K\in V(B(G))$. We show that
\[
d_{B(G)}(H,K)\leq 3.
\]

Suppose first that $H\cap K\neq\{e\}$. If $H$ and $K$ are
incomparable, then $H\sim K$. Suppose, say, that $H<K$. Since $H$ is
a vertex of $B(G)$, choose $X\sim H$. Then $X$ and $K$ are
incomparable. Indeed, if $X\leq K$, then $H$ and $X$ are subgroups
of the cyclic group $K$ and hence are comparable, contrary to
$H\sim X$; while $K\leq X$ would imply $H\leq X$, again a
contradiction. Moreover,
\[
\{e\}\neq H\cap X\leq K\cap X.
\]
Hence $X\sim K$, and therefore
\[
H\sim X\sim K.
\]

It remains to consider the case
\[
H\cap K=\{e\}.
\]
The two subgroups cannot both have even order, since every cyclic
subgroup of even order contains the unique involution $P$. Thus,
after interchanging $H$ and $K$ if necessary, we may assume that
$H$ has odd order.

Choose a subgroup $Q<H$ of prime order. The containment is proper,
since a subgroup of prime order is not a vertex of $B(G)$. Since
$P$ is central, $PQ$ is cyclic, and
\[
H\sim PQ.
\]

If $K$ has even order, then $P\leq K$. Since $H\cap K=\{e\}$,
we have $Q\nleq K$. Thus $PQ$ and $K$ are incomparable and have the
non-trivial common subgroup $P$, so
\[
PQ\sim K.
\]
Hence
\[
H\sim PQ\sim K.
\]

Finally, suppose that $K$ also has odd order. Choose a subgroup
$R<K$ of prime order. As above,
\[
K\sim PR.
\]
Since $H\cap K=\{e\}$, we have $Q\neq R$. The cyclic subgroups
$PQ$ and $PR$ are distinct, contain $P$, and are incomparable;
hence
\[
PQ\sim PR.
\]
Therefore
\[
H\sim PQ\sim PR\sim K.
\]

Thus every two vertices of $B(G)$ are at distance at most $3$, and
so $B(G)$ is connected with
\[
\diam B(G)\leq 3.
\]
The assertion for $D(G)$ follows from Corollary~\ref{cor:D-transfer}.
\end{proof}

\begin{example}[Sharpness of the bound $3$]\label{ex:c2c9c9inv}
Let
\[
G=C_2\times C_9\times C_9
 =\langle c\rangle\times\langle a\rangle\times\langle b\rangle,
\]
where
\[
|c|=2,\qquad |a|=|b|=9.
\]
Then $c$ is the unique involution of $G$.

Put
\[
H=\langle a\rangle,\qquad K=\langle b\rangle,
\]
and
\[
A=\langle a^3c\rangle,\qquad B=\langle b^3c\rangle.
\]
The subgroups $A$ and $B$ are cyclic of order $6$, and
\[
H\cap A=\langle a^3\rangle,\qquad
A\cap B=\langle c\rangle,\qquad
B\cap K=\langle b^3\rangle.
\]
In each case the two subgroups are incomparable. Hence
\[
H\sim A\sim B\sim K,
\]
and therefore
\[
d_{B(G)}(H,K)\leq 3.
\]
Since $H\cap K=\{e\}$, we have $H\not\sim K$. 
We show that $H$ and $K$ have no common neighbour. Suppose that
$X\in V(B(G))$ is adjacent to both. Since
\[
H\cap X\neq\{e\},
\]
the unique subgroup $\langle a^3\rangle$ of order $3$ in $H$ is
contained in $X$. Similarly,
\[
\langle b^3\rangle\leq X.
\]
But $\langle a^3\rangle$ and $\langle b^3\rangle$ are distinct
subgroups of order $3$, whereas a cyclic group has a unique subgroup
of each order. This is impossible. Thus
\[
d_{B(G)}(H,K)>2.
\]
Consequently,
\[
d_{B(G)}(H,K)=3,
\]
and hence
\[
\diam B(G)=3.
\]
By Corollary~\ref{cor:D-transfer},
\[
\diam D(G)=3.
\]
Thus the bound in the preceding theorem is sharp (see Figure~\ref{fig:invc299}).
\end{example}

\begin{figure}[hb]
\begin{center}
\includegraphics[scale=0.5]{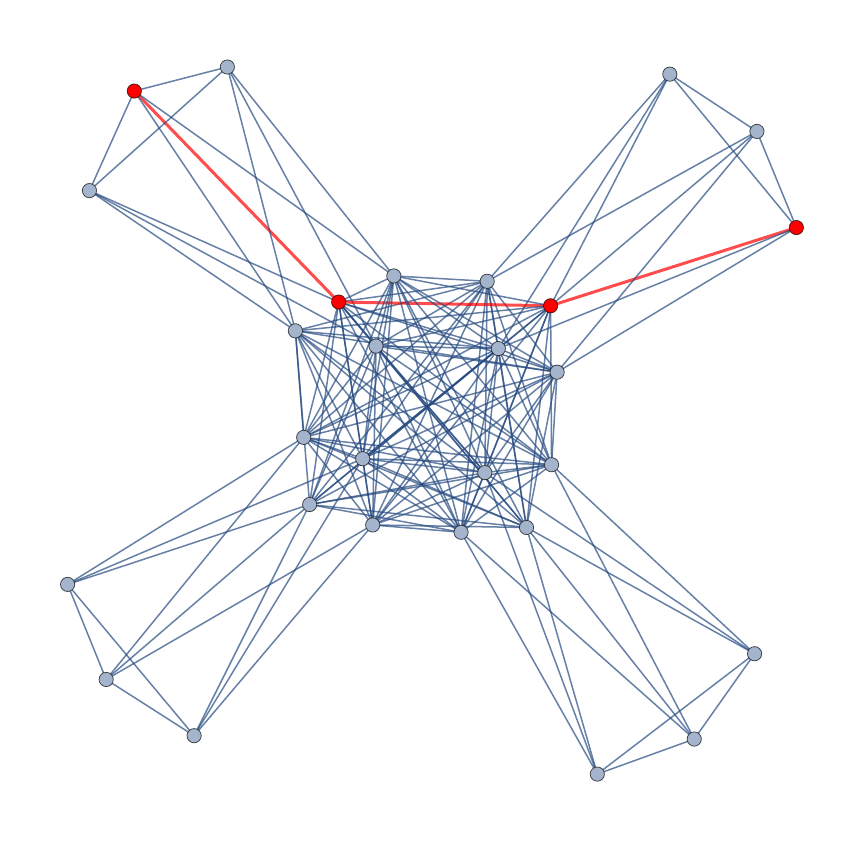}
\caption{The graph $B(C_2\times C_9\times C_9)$ in Example~\ref{ex:c2c9c9inv}. A diameter path is highlighted in red.}\label{fig:invc299}
\end{center}
\end{figure}


\vspace{-2em}
\section{Non-abelian groups with trivial center}
\label{sec:centerless}

Throughout this section, $G$ is a finite non-abelian group with
\[
    Z(G)=\{e\}.
\]
Since finite $p$-groups have non-trivial centers, $G$ is not a $p$-group and so $|\pi(G)|\geq 2$.

We use the notions of branching and $b$-normality from
Definitions~6.4 and~6.5. Thus a subgroup $P\leq G$ of prime order is
branching if it is contained in the intersection of two incomparable cyclic
subgroups. By Proposition~\ref{prop:branching}, this is equivalent to $P$ being $b$-normal in
$N_G(P)$. Let
\[
    \operatorname{Br}(G)
    =
    \{P\leq G: |P|\text{ is prime and }P\text{ is branching}\}.
\]
By Corollary~\ref{cor:nonempty1}, $B(G)$ is non-empty if and only if
$\operatorname{Br}(G)\neq\varnothing$.

We first isolate what happens when a branching subgroup has a $p$-group
centralizer. The resulting single-branching case is the centerless analogue
of the branching criterion for $p$-groups.

\begin{proposition}[The single-branching case]
\label{prop:centerless-single-branching}
Let $P\in\operatorname{Br}(G)$ have order $p$.
\begin{enumerate}
\item If $C_G(P)$ is a $p$-group, then the set
\[
    \mathcal C_P
    =
    \{L\in V(B(G)):P\leq L\}
\]
is a connected component of $B(G)$ and
\[
    \operatorname{diam}\mathcal C_P\leq2.
\]
In particular, if $G$ has another branching prime-order subgroup
$R\neq P$, then $B(G)$ is disconnected.

\item If $P$ is the unique branching prime-order subgroup of $G$, then
$B(G)$ is connected and
\[
    \operatorname{diam}B(G)\leq2,
    \qquad
    \operatorname{diam}D(G)=2.
\]
\end{enumerate}
\end{proposition}

\begin{proof}
Suppose first that $C_G(P)$ is a $p$-group. If $L\in\mathcal C_P$, then
$L\leq C_G(P)$, so $L$ is a cyclic $p$-group. Let $M\sim L$. Since every
non-trivial subgroup of the cyclic $p$-group $L$ contains its unique
subgroup $P$ of order $p$, we have
\[
    P\leq L\cap M.
\]
Thus $M\leq C_G(P)$ and $P\leq M$, so $M\in\mathcal C_P$. Hence no edge
of $B(G)$ leaves $\mathcal C_P$.

We show that $\mathcal C_P$ is connected with diameter at most two. Let
$H,K\in\mathcal C_P$. If they are incomparable, then $H\sim K$. Suppose,
say, that $H<K$. Since $H$ is a vertex of $B(G)$, choose $L\sim H$.
By the preceding paragraph, $L\in\mathcal C_P$. The subgroup $L$ cannot
be contained in $K$, for then $H$ and $L$ would be subgroups of the cyclic
group $K$ and hence comparable. Also $K\nleq L$, since otherwise
$H<K\leq L$, contrary to $H\sim L$. Therefore $L$ and $K$ are
incomparable and contain $P$, so
\[
    H\sim L\sim K.
\]
Thus $\mathcal C_P$ is a connected component and
$\operatorname{diam}\mathcal C_P\leq2$.

If $R\neq P$ is another branching prime-order subgroup, choose an edge
$A\sim B$ with $R\leq A\cap B$. Neither $A$ nor $B$ can lie in
$\mathcal C_P$: a cyclic $p$-group containing $P$ cannot contain a
distinct subgroup of order $p$, nor can it contain a subgroup of prime
order different from $p$. Hence $B(G)$ has a component different from
$\mathcal C_P$ and is disconnected. This proves~(1).

For~(2), assume that $P$ is the unique branching prime-order subgroup.
Let $L\in V(B(G))$ and choose $T\sim L$. Any prime-order subgroup of
$L\cap T$ is branching, and therefore it is $P$. Hence every vertex of
$B(G)$ contains $P$. Let $H,K\in V(B(G))$. If they are incomparable, then
$H\sim K$. If, say, $H<K$, choose $L\sim H$. Again $P\leq L$. As above,
$L\nleq K$ and $K\nleq L$, so $H\sim L\sim K$. Thus
\[
    \operatorname{diam}B(G)\leq2.
\]
The assertion for $D(G)$ follows from Corollary~\ref{cor:D-transfer}.
\end{proof}

Now, when the center is non-trivial, the arguments of the preceding section are
organized around a central hub. No such subgroup is available when
$Z(G)=\{e\}$. To analyze the mixed-prime interaction between branching
subgroups, we therefore attach a local hub to each branching subgroup.

\begin{definition}[Local hub]
\label{def:local-hub}
Let $P\in\operatorname{Br}(G)$ have order $p$. Define
\[
    \mathcal H(P)
    =
    \{PQ: Q\leq C_G(P),\ |Q|=q\text{ for some prime }q\neq p\}.
\]
Since $P$ and $Q$ commute and have coprime orders, every member $PQ$ of
$\mathcal H(P)$ is cyclic of order $pq$. We call $\mathcal H(P)$ the
\emph{local hub at $P$}.
\end{definition}

By Cauchy's theorem, $\mathcal H(P)\neq\varnothing$ if and only if
$C_G(P)$ is not a $p$-group. The next lemma shows that a non-empty local hub
really lies inside $B(G)$, even when it consists of a single cyclic subgroup. For subgroups $P,R\leq G$, we write $[P,R]=1$ if every element of $P$ commutes with every element of $R$.

\begin{lemma}[Local-hub lemma]
\label{lem:local-hub}
Let $P\in\operatorname{Br}(G)$ have order $p$, and suppose that
$C_G(P)$ is not a $p$-group. Then every member of $\mathcal H(P)$ is a
vertex of $B(G)$, and $\mathcal H(P)$ induces a clique in $B(G)$.
\end{lemma}

\begin{proof}
Let $X=PQ\in\mathcal H(P)$, where $|Q|=q\neq p$. Since $P$ is branching,
choose incomparable cyclic subgroups $A,B\leq G$ such that
\[
    P\leq A\cap B.
\]
We show first that $X\in V(B(G))$.

If $X\nleq A$, then $A\nleq X$: indeed, a cyclic group of order $pq$ has
only $P$ and $X$ among its cyclic subgroups which contain $P$, whereas
$A>P$. Hence $A\sim X$. Similarly, if $X\leq A$ but $X\nleq B$, then
$B\sim X$.

It remains to consider the case
\[
    X\leq A\cap B.
\]
Since $A\nleq B$, some Sylow $r$-subgroup $S$ of the cyclic group $A$ is
not contained in $B$. If $r=p$ or $r=q$, then $X\cap S\neq\{e\}$.
Moreover, $S\nleq X$, since $X\leq B$ but $S\nleq B$, while $X\nleq S$
because $X$ has order divisible by two distinct primes. Thus $X\sim S$.

Suppose now that $r\notin\{p,q\}$. Since $P,S\leq A$, the subgroup
$PS$ is cyclic. The cyclic subgroups $X$ and $PS$ contain $P$ and
are incomparable, so
\[
X\sim PS.
\]
Thus in every case $X$ has a neighbour, and therefore
$X\in V(B(G))$.

Finally, let $PQ_1$ and $PQ_2$ be distinct members of $\mathcal H(P)$.
They have the non-trivial common subgroup $P$. If $|Q_1|\neq|Q_2|$, then
 orders of $PQ_1$ and $PQ_2$ are $p|Q_1|$ and $p|Q_2|$. So neither can contain the other.
If $|Q_1|=|Q_2|$, then they are distinct cyclic subgroups of the same order
and are again incomparable. Thus
\[
    PQ_1\sim PQ_2,
\]
and $\mathcal H(P)$ induces a clique.
\end{proof}

The intersections of local hubs have a particularly simple form.

\begin{lemma}[Intersection of local hubs]
\label{lem:hub-intersection}
Let $P,R\in\operatorname{Br}(G)$ be distinct subgroups, with
$|P|=p$ and $|R|=r$, and suppose that both local hubs are non-empty. Then
\[
    \mathcal H(P)\cap\mathcal H(R)\neq\varnothing
\]
if and only if
\[
    p\neq r
    \qquad\text{and}\qquad
    [P,R]=1.
\]
In this case
\[
    \mathcal H(P)\cap\mathcal H(R)=\{PR\}.
\]
\end{lemma}

\begin{proof}
Suppose that $X\in\mathcal H(P)\cap\mathcal H(R)$. Then $X$ is cyclic and
contains both $P$ and $R$. If $p=r$, the cyclic group $X$ would contain two
distinct subgroups of order $p$, which is impossible. Hence $p\neq r$.
Since $X$ is cyclic, $P$ and $R$ commute, and the subgroup $PR$ is cyclic
of order $pr$. But every member of $\mathcal H(P)$ has order the product of
$p$ and one other prime. Since it contains $R$, that other prime must be
$r$. Therefore $X=PR$.

Conversely, if $p\neq r$ and $[P,R]=1$, then $PR$ is cyclic of order $pr$.
Moreover,
\[
    R\leq C_G(P)
    \qquad\text{and}\qquad
    P\leq C_G(R),
\]
so
\[
    PR\in\mathcal H(P)\cap\mathcal H(R).
\]
The uniqueness assertion follows from the first part.
\end{proof}

For the rest of this section we use the following mixed-centralizer
hypothesis:
\begin{equation*}\tag{MC}\label{eq:MC}
    C_G(P)\text{ is not a $p$-group for every }
    P\in\operatorname{Br}(G)\text{ of order }p.
\end{equation*}
Under~\eqref{eq:MC}, every branching subgroup has a non-empty local hub.
This allows us to record how the local hubs overlap.

\begin{definition}[Auxiliary hub graph]
\label{def:auxiliary-hub-graph}
Assume~\eqref{eq:MC}. The \emph{auxiliary hub graph} $\mathcal A(G)$ is the
graph with vertex set
\[
    V(\mathcal A(G))=\operatorname{Br}(G),
\]
in which two distinct vertices $P$ and $R$ are adjacent if and only if
\[
    \mathcal H(P)\cap\mathcal H(R)\neq\varnothing.
\]
Equivalently, by Lemma~\ref{lem:hub-intersection},
\[
    P\sim_{\mathcal A(G)}R
    \quad\Longleftrightarrow\quad
    |P|\neq|R|\text{ and }[P,R]=1.
\]
\end{definition}

The construction of $\mathcal A(G)$ is naturally linked with $B(G)$. Let
$\mathcal B_{\text{mix}}(G)$ be the subgraph of $B(G)$ induced by the cyclic
subgroups
\[
    PR,\qquad P,R\in\operatorname{Br}(G),\quad
    |P|\neq |R|,\quad [P,R]=1.
\]
By Lemma~\ref{lem:local-hub}, these are vertices of $B(G)$, and the
correspondence
\[
    \{P,R\}\longmapsto PR
\]
identifies $\mathcal B_{\text{mix}}(G)$ with the line graph of $\mathcal A(G)$:
\[
    \mathcal B_{\text{mix}}(G)\cong L(\mathcal A(G)),
\]
where $L(\mathcal{A}(G))$ is the line graph of $\mathcal{A}(G)$.

Thus, if $\mathcal A(G)$ is connected and has at least two vertices,
$\mathcal B_{\text{mix}}(G)$ is a connected induced subgraph of $B(G)$.

For $L\in V(B(G))$, put
\[
    S(L)=\{P\in\operatorname{Br}(G):P\leq L\}.
\]
The set $S(L)$ is non-empty: if $T\sim L$, any prime-order subgroup of
$L\cap T$ is branching. Moreover, $S(L)$ is a clique in $\mathcal A(G)$,
since two distinct prime-order subgroups of the cyclic group $L$ have
different orders and commute.

The next lemma is the local replacement for the central-hub distance
arguments of Section~6.

\begin{lemma}[Attachment to a local hub]
\label{lem:hub-attachment}
Assume~\eqref{eq:MC}. Let $L,T\in V(B(G))$ with $L\sim T$, and let
$P\leq L\cap T$ have prime order. Then $P\in\operatorname{Br}(G)$.
Moreover, for every $X\in\mathcal H(P)$,
\[
    d_{B(G)}(L,X)\leq 2
    \qquad\text{and}\qquad
    d_{B(G)}(T,X)\leq 2.
\]
\end{lemma}

\begin{proof}
The subgroup $P$ is branching because the incomparable cyclic subgroups
$L$ and $T$ both contain it. Write $|P|=p$, and let
\[
    X=PQ\in\mathcal H(P),
    \qquad |Q|=q\neq p.
\]
We prove the assertion for $L$; the argument for $T$ is symmetric.

If $Q\nleq L$, then $P\leq L\cap X$, while neither $L$ nor $X$ contains
the other. Hence $L\sim X$.

Suppose that $Q\leq L$. If $Q\nleq T$, then similarly $T\sim X$, and
\[
    L\sim T\sim X.
\]
We may therefore assume that
\[
    Q\leq L\cap T,
\]
so $X\leq L\cap T$. Since $T\nleq L$, choose a Sylow $r$-subgroup $S$ of
the cyclic group $T$ such that $S\nleq L$.

If $r=p$ or $r=q$, then $S\cap L\neq\{e\}$ and $S\cap X\neq\{e\}$.
Moreover, $S\nleq L$ by choice, while $L\nleq S$ because $L$ contains the
other prime-order subgroup among $P$ and $Q$. Also $S\nleq X$, since
$X\leq L$, and $X\nleq S$ because $X$ has two distinct prime divisors.
Consequently
\[
    L\sim S\sim X.
\]

Finally, suppose that $r\notin\{p,q\}$. Put $Y=PS$. Since $P,S\leq T$,
the subgroup $Y$ is cyclic. We have $P\leq L\cap Y$, while $Y\nleq L$
because $S\nleq L$, and $L\nleq Y$ because $Q\leq L$ and
$q\nmid |Y|$. Thus $L\sim Y$. Similarly, $X$ and $Y$ contain $P$ and are
incomparable, so $Y\sim X$. Hence again
\[
    d_{B(G)}(L,X)\leq 2.
\]

\vspace{-2em}
\end{proof}

We now obtain the connectivity criterion under the mixed-centralizer condition.

\begin{theorem}[Connectivity under (MC)]
\label{thm:centerless-connectivity}
Let $G$ be a finite non-abelian group with 
trivial center. Assume that $B(G)$ is non-empty and that~\eqref{eq:MC}
holds. Then
\[
    B(G)\text{ is connected}
    \quad\Longleftrightarrow\quad
    \mathcal A(G)\text{ is connected}.
\]
Whenever these equivalent conditions hold,
\[
    \operatorname{diam}\mathcal A(G)-1
    \leq
    \operatorname{diam}B(G)
    \leq
    \max\{4,\operatorname{diam}\mathcal A(G)+1\}.
\]
Consequently $D(G)$ is connected and
\[
    \operatorname{diam}D(G)
    =
    \max\{2,\operatorname{diam}B(G)\}
    \leq
    \max\{4,\operatorname{diam}\mathcal A(G)+1\}.
\]
In particular, if $\operatorname{diam}\mathcal A(G)\geq3$, then
\[
    \operatorname{diam}B(G),\operatorname{diam}D(G)
    \leq
    \operatorname{diam}\mathcal A(G)+1.
\]
\end{theorem}

\begin{proof}
Suppose first that $B(G)$ is connected. Let
$P,R\in\operatorname{Br}(G)$. By Lemma~\ref{lem:local-hub}, choose
\[
    X\in\mathcal H(P),
    \qquad
    Y\in\mathcal H(R),
\]
with $X,Y\in V(B(G))$. Since $B(G)$ is connected, there is an $X$--$Y$
path
\[
    X=L_0\sim L_1\sim\cdots\sim L_m=Y.
\]
If $m=0$, then $X=Y$, so the two local hubs intersect and
$P\sim_{\mathcal A(G)}R$.

Assume $m\geq1$. For each $i\in\{1,\ldots,m\}$ choose a prime-order
subgroup
\[
    P_i\leq L_{i-1}\cap L_i.
\]
Each $P_i$ is branching. Since $L_i$ is cyclic, the two subgroups
$P_i$ and $P_{i+1}$ either coincide or have distinct prime orders and
commute. Thus, after deleting repetitions,
\[
    P_1,P_2,\ldots,P_m
\]
is a walk in $\mathcal A(G)$.

Now $L_0=X$ has order the product of two distinct primes and contains $P$.
Hence either $P_1=P$, or $P_1$ is the other prime-order subgroup of $X$;
in the latter case $P\sim_{\mathcal A(G)}P_1$. Similarly, either
$P_m=R$, or $P_m\sim_{\mathcal A(G)}R$. Therefore $P$ and $R$ lie in the
same component of $\mathcal A(G)$. Hence $\mathcal A(G)$ is connected.
The same argument also gives
\begin{equation}\label{eq:ABdistance}
d_{\mathcal A(G)}(P,R)\leq d_{B(G)}(X,Y)+1.
\end{equation}

Conversely, suppose that $\mathcal A(G)$ is connected. If
$|V(\mathcal A(G))|=1$, say $V(\mathcal A(G))=\{P\}$, then every edge
$L\sim T$ of $B(G)$ branches over $P$. Fix $X\in\mathcal H(P)$. By
Lemma~\ref{lem:hub-attachment}, every vertex of $B(G)$ has distance at most
$2$ from $X$. Hence $B(G)$ is connected and
\[
    \operatorname{diam}B(G)\leq4.
\]

Assume now that $|V(\mathcal A(G))|\geq2$. Then
$\mathcal B_{\text{mix}}(G)\cong L(\mathcal A(G))$ is a connected induced subgraph
of $B(G)$. Let $L\in V(B(G))$, choose $T\sim L$, and choose a prime-order
subgroup $P\leq L\cap T$. Then $P\in\operatorname{Br}(G)$. Since
$\mathcal A(G)$ is connected and has at least two vertices, $P$ has a
neighbour $R$ in $\mathcal A(G)$. Thus
\[
    PR\in\mathcal B_{\text{mix}}(G)\cap\mathcal H(P),
\]
and Lemma~\ref{lem:hub-attachment} gives
\[
    d_{B(G)}(L,PR)\leq2.
\]
Therefore every vertex of $B(G)$ is joined to the connected subgraph
$\mathcal B_{\text{mix}}(G)$, and hence $B(G)$ is connected.

For the diameter bound, let $L_1,L_2\in V(B(G))$, and put
\[
    S_i=S(L_i)\qquad(i=1,2).
\]
Suppose first that $S_1\cap S_2=\varnothing$. Let
\[
    P_0\sim_{\mathcal A(G)}P_1
      \sim_{\mathcal A(G)}\cdots
      \sim_{\mathcal A(G)}P_k
\]
be a shortest path in $\mathcal A(G)$ from $S_1$ to $S_2$, where
$P_0\in S_1$ and $P_k\in S_2$. Then $k\geq1$, and the minimality of
the path gives
\[
    P_1\notin S_1,
    \qquad
    P_{k-1}\notin S_2
\]
(with the evident interpretation when $k=1$). Hence
\[
    L_1\sim P_0P_1,
    \qquad
    P_{k-1}P_k\sim L_2.
\]
The successive edges of the above path correspond in
$L(\mathcal A(G))\cong\mathcal B_{\text{mix}}(G)$ to the vertices
\[
    P_0P_1,\ P_1P_2,\ldots,\ P_{k-1}P_k,
\]
which form a path of length $k-1$. Therefore
\[
    d_{B(G)}(L_1,L_2)
    \leq
    1+(k-1)+1
    =k+1
    \leq
    \operatorname{diam}\mathcal A(G)+1.
\]

Now suppose that $S_1\cap S_2\neq\varnothing$, and choose
$P\in S_1\cap S_2$. Fix any $X\in\mathcal H(P)$. Applying
Lemma~\ref{lem:hub-attachment} to an edge incident with $L_1$ and to an
edge incident with $L_2$ gives
\[
    d_{B(G)}(L_1,X)\leq2,
    \qquad
    d_{B(G)}(L_2,X)\leq2.
\]
Consequently
\[
    d_{B(G)}(L_1,L_2)\leq4.
\]
Combining the two cases yields
\[
    \operatorname{diam}B(G)
    \leq
    \max\{4,\operatorname{diam}\mathcal A(G)+1\}.
\]

For the lower bound, choose $P,R\in V(\mathcal A(G))$ such that
\[
    d_{\mathcal A(G)}(P,R)=\operatorname{diam}\mathcal A(G),
\]
and choose $X\in\mathcal H(P)$ and $Y\in\mathcal H(R)$. Applying
(\ref{eq:ABdistance}) to a shortest $X$--$Y$ path in $B(G)$ gives
\[
    \operatorname{diam}\mathcal A(G)
    \leq
    d_{B(G)}(X,Y)+1
    \leq
    \operatorname{diam}B(G)+1,
\]
which yields
\[
    \operatorname{diam}\mathcal A(G)-1
    \leq
    \operatorname{diam}B(G).
\]
The assertions concerning $D(G)$ follow from Corollary~\ref{cor:D-transfer}.
\end{proof}

Combining Proposition~\ref{prop:centerless-single-branching} with
Theorem~\ref{thm:centerless-connectivity} gives the complete connectivity
criterion for the centerless case.

\begin{corollary}[Complete centerless connectivity criterion]
\label{cor:centerless-complete}
Assume that $B(G)$ is non-empty. Then $B(G)$ is connected if and only if
exactly one of the following alternatives applies:
\begin{enumerate}
\item $\operatorname{Br}(G)=\{P\}$ consists of a single subgroup;
\item $|\operatorname{Br}(G)|\geq2$, condition~\eqref{eq:MC} holds, and
$\mathcal A(G)$ is connected.
\end{enumerate}
In the first case,
\[
    \operatorname{diam}B(G)\leq2,
    \qquad
    \operatorname{diam}D(G)=2.
\]
In the second case,
\[
    \operatorname{diam}\mathcal A(G)-1
    \leq
    \operatorname{diam}B(G)
    \leq
    \max\{4,\operatorname{diam}\mathcal A(G)+1\},
\]
and
\[
    \operatorname{diam}D(G)
    =
    \max\{2,\operatorname{diam}B(G)\}
    \leq
    \max\{4,\operatorname{diam}\mathcal A(G)+1\}.
\]
\end{corollary}

\begin{proof}
Suppose that $B(G)$ is connected. If $\operatorname{Br}(G)$ consists of a
single subgroup, then~(1) holds. Assume instead that
$|\operatorname{Br}(G)|\geq2$. If condition~\eqref{eq:MC} failed, there
would exist $P\in\operatorname{Br}(G)$ of order $p$ such that $C_G(P)$ is
a $p$-group. Proposition~\ref{prop:centerless-single-branching} would then
force $B(G)$ to be disconnected, a contradiction. Hence~\eqref{eq:MC}
holds, and Theorem~\ref{thm:centerless-connectivity} implies that
$\mathcal A(G)$ is connected. Thus~(2) holds.

Conversely, if~(1) holds, Proposition~\ref{prop:centerless-single-branching}
gives that $B(G)$ is connected. If~(2) holds,
Theorem~\ref{thm:centerless-connectivity} gives that $B(G)$ is connected.
The diameter assertions are exactly those of the proposition and theorem.
\end{proof}

\begin{remark}
\label{rem:plus-one-small-diameter}
The term $4$ in the upper bound cannot simply be omitted. Let
\[
    G=D_{1350}=\langle r,s:r^{675}=s^2=e,\ srs=r^{-1}\rangle.
\]
Since $675$ is odd, $Z(G)=\{e\}$. The only branching prime-order
subgroups are the unique subgroups of $\langle r\rangle$ of orders $3$
and $5$, and hence
\[
    \mathcal A(G)\cong K_2,
    \qquad
    \operatorname{diam}\mathcal A(G)=1.
\]
Every cyclic subgroup not contained in $\langle r\rangle$ has order $2$,
so
\[
    B(G)\cong B(C_{675}).
\]
The latter graph of the cyclic group of order $675=3^35^2$ has diameter $3$; for example, its subgroups of orders
$9$ and $225$ have distance $3$, with a path through subgroups of orders
\[
    9=3^2,\ 15=3\cdot 5,\ 27=3^3,\ 225=3^2\cdot 5^2,
\]
and they have no common neighbour (see Figure~\ref{fig:d1350}). Thus
\[
    \operatorname{diam}B(G)=3
    >
    \operatorname{diam}\mathcal A(G)+1=2.
\]
\end{remark}

\begin{figure}[h]
\begin{center}
\includegraphics[scale=0.3]{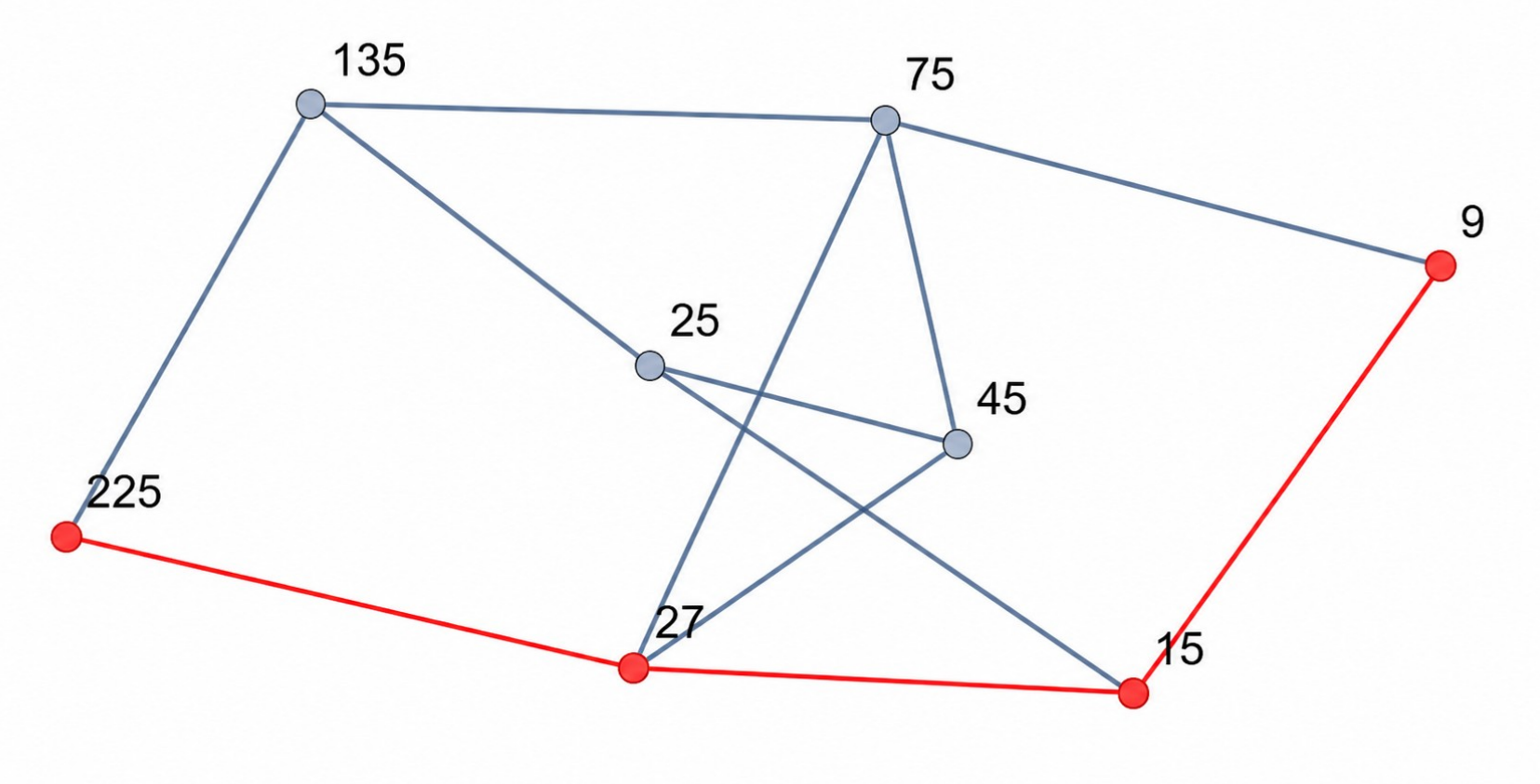}
\caption{The graph $B(D_{1350})$. The vertex labels indicate the order of subgroups. A diameter path is highlighted in red.}\label{fig:d1350}
\end{center}
\end{figure}

The next examples illustrate the empty case and the two mixed-hub possibilities: the local hubs may fail to connect to one another, or the auxiliary graph may connect all local hubs.

\begin{example}[An empty graph]
\label{ex:centerless-s3}
Let $G=S_3$. Then $Z(G)=\{e\}$, and every non-trivial cyclic subgroup of
$G$ has prime order. Hence no prime-order subgroup is branching. By
Corollary~\ref{cor:nonempty1},
\[
    B(S_3)=\varnothing=D(S_3).
\]
Thus non-emptiness is not automatic for centerless non-abelian groups.
\end{example}

\begin{example}[Disconnected local hubs]
\label{ex:a4-square}
Let
\[
    G=A_4\times A_4.
\]
Since $Z(A_4)=\{e\}$, we have $Z(G)=\{e\}$. Recall that $A_4$ has three
subgroups of order $2$ and four subgroups of order $3$, and
\[
    C_{A_4}(P)\cong V_4= C_2\times C_2
    \quad\text{if }|P|=2,
    \qquad
    C_{A_4}(P)=P
    \quad\text{if }|P|=3.
\]

A prime-order subgroup of $G$ which projects non-trivially onto both direct
factors is not branching. Indeed, its centralizer is $V_4\times V_4$ in
the order-$2$ case and $C_3\times C_3$ in the order-$3$ case; these groups
have exponent $2$ and $3$, respectively, so such a subgroup has no proper
cyclic overgroup.

\begin{figure}[ht]
\begin{center}
\includegraphics[scale=0.5]{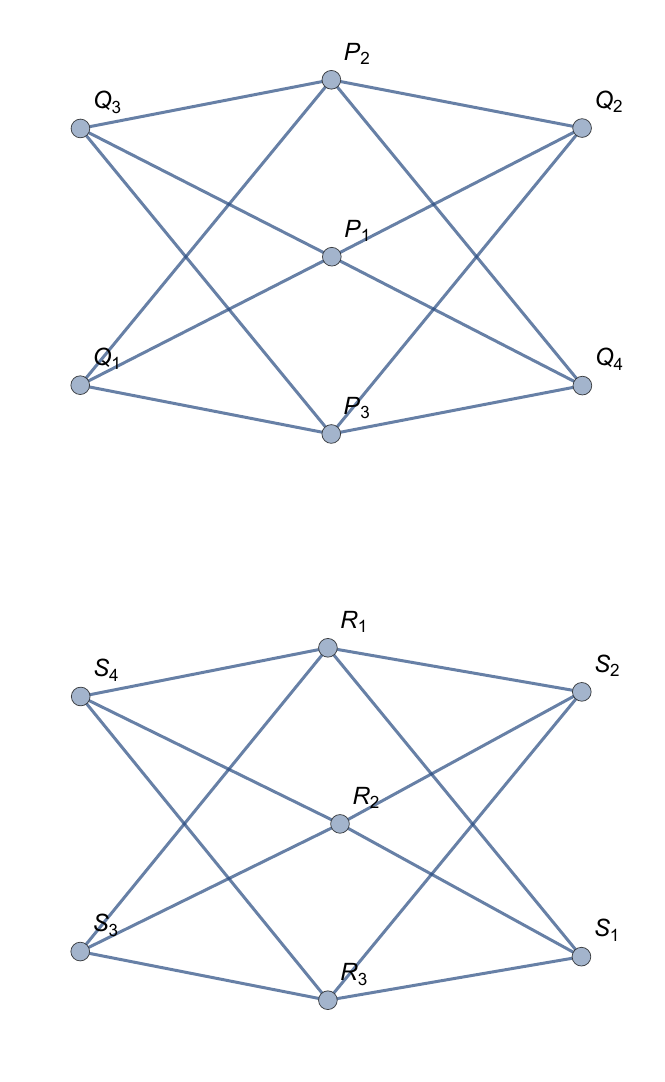} \qquad \includegraphics[scale=0.5]{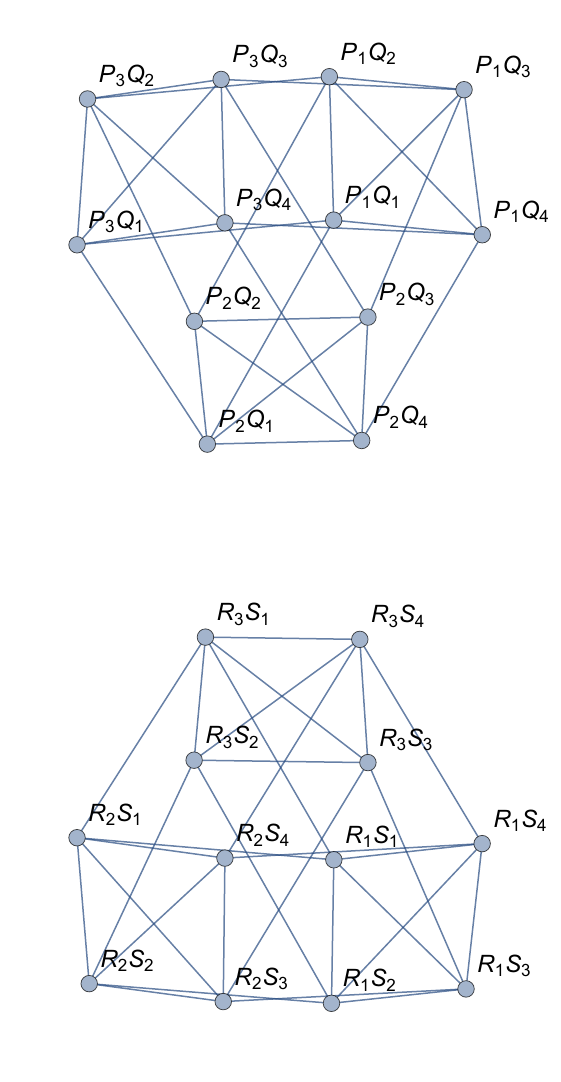}
\caption{The graphs $\mathcal{A}(G)$ (left) and $B(G)$ (right) for $G=A_4\times A_4$.}\label{fig:a4a4abg}
\end{center}
\end{figure}

On the other hand, every prime-order subgroup contained in exactly one
direct factor is branching. For example, if
\[
    P\leq A_4\times\{e\},
    \qquad |P|=2,
\]
then the second factor contains four distinct subgroups $Q$ of order $3$,
and the corresponding cyclic subgroups $PQ$ of order $6$ are distinct
members of $\mathcal H(P)$. The other cases are analogous. Hence
$\operatorname{Br}(G)$ consists exactly of the prime-order subgroups
supported in one factor, and condition~\eqref{eq:MC} holds.

A subgroup of order $2$ in one factor commutes with every subgroup of order
$3$ in the other factor, while no non-trivial subgroup of order $2$ in
$A_4$ commutes with a subgroup of order $3$ in the same factor. Therefore
\[
    \mathcal A(G)
    \cong
    K_{3,4}\mathbin{\dot\cup}K_{3,4}.
\]
Thus $\mathcal A(G)$ is disconnected, and
Theorem~\ref{thm:centerless-connectivity} gives that $B(G)$ and $D(G)$ are
disconnected.

Here $P_1,P_2,P_3$ and $S_1,\ldots,S_4$ denote, respectively, the
subgroups of orders $2$ and $3$ in the first $A_4$-factor, while
$R_1,R_2,R_3$ and $Q_1,\ldots,Q_4$ denote those in the second factor.
Figure~\ref{fig:a4a4abg} shows that under the correspondence
\[
\{P_i,Q_j\}\longmapsto P_iQ_j,\qquad
\{R_i,S_j\}\longmapsto R_iS_j,
\]
each edge of $\mathcal A(G)$ corresponds to the associated vertex of
$B(G)$ and vertices of $B(G)$ are adjacent if and only if the corresponding edges in $\mathcal A(G)$ have common endpoint. Thus, in this case, $B(G)$ is the line graph of $\mathcal A(G)$. Moreover, since $L(K_{3,4})\cong K_3\square K_4$,
\[
B(G)\cong (K_3\square K_4)\,\dot\cup\,(K_3\square K_4),
\]
where $X\square Y$ is the Cartesian product of two graphs $X$ and $Y$.
\end{example}

\begin{example}[Connected local hubs]
\label{ex:a4-cube}
Let
\[
    G=A_4\times A_4\times A_4.
\]
Again $Z(G)=\{e\}$. For a subgroup $P\leq G$ of prime order, define its
support by
\[
    \operatorname{supp}(P)
    =
    \{i\in\{1,2,3\}: \pi_i(P)\neq\{e\}\},
\]
where $\pi_i$ is the projection onto the $i$-th factor.

If $\operatorname{supp}(P)=\{1,2,3\}$, then $C_G(P)$ is $V_4^3$ when
$|P|=2$ and $C_3^3$ when $|P|=3$. In either case the centralizer has
exponent $|P|$, so $P$ is not branching.

Suppose instead that $\operatorname{supp}(P)$ is a proper non-empty subset
of $\{1,2,3\}$. Choose a coordinate outside the support. If $|P|=2$, that
unused factor contains four distinct subgroups of order $3$; if $|P|=3$,
it contains three distinct subgroups of order $2$. Choosing two such
subgroups $Q_1,Q_2$ gives distinct cyclic subgroups
\[
    PQ_1,
    \qquad
    PQ_2,
\]
of order $6$ which both contain $P$. Hence $P$ is branching. Thus the
branching prime-order subgroups are precisely those whose support has size
$1$ or $2$, and each of their centralizers contains a full unused copy of
$A_4$. In particular, condition~\eqref{eq:MC} holds.

Let $P,R\in\operatorname{Br}(G)$ have different prime orders. Since no
non-trivial element of order $2$ in $A_4$ commutes with an element of order
$3$, we have
\[
    [P,R]=1
    \quad\Longleftrightarrow\quad
    \operatorname{supp}(P)\cap\operatorname{supp}(R)=\varnothing.
\]
Hence this is also the adjacency condition in $\mathcal A(G)$.

Every vertex $P$ of $\mathcal A(G)$ omits at least one coordinate. Choosing
a subgroup $R$ of the other prime order supported on such an omitted
coordinate gives
\[
    P\sim_{\mathcal A(G)}R.
\]
Thus every vertex is adjacent to a vertex supported on a single coordinate.
The single-coordinate vertices form one connected component: vertices of
different prime orders on different coordinates are adjacent; two
single-coordinate vertices of the same prime order have a common neighbour
of the other prime order on a third coordinate; and if they have different
prime orders but the same coordinate, the other two coordinates give a path
of length three between them. Therefore $\mathcal A(G)$ is connected.

We next determine its diameter. Any two single-coordinate vertices have
distance at most $3$, and every vertex is either single-coordinate or is
adjacent to one. Hence
\[
    \operatorname{diam}\mathcal A(G)\leq5.
\]
To see that equality holds, choose a subgroup $P$ of order $2$ and a subgroup
$R$ of order $3$ with
\[
    \operatorname{supp}(P)
    =
    \operatorname{supp}(R)
    =
    \{1,2\}.
\]
Since every edge of $\mathcal A(G)$ joins a subgroup of order $2$
to a subgroup of order $3$, the graph $\mathcal A(G)$ is bipartite.
Hence every $P$--$R$ path has odd length. There is no path of length $3$: the first intermediate vertex would
have to be supported on $\{3\}$, leaving no non-empty support for the
second intermediate vertex which is disjoint from both $\{3\}$ and
$\{1,2\}$. On the other hand, a path of length $5$ is obtained by using
successively supports
\[
\{1,2\},\{3\},\{1\},\{2\},\{3\},\{1,2\},
\]
and choosing the corresponding vertices to have orders
$2,3,2,3,2,3$, respectively.
\[
    \operatorname{diam}\mathcal A(G)=5.
\]

Every element of $A_4^3$ has order $1$, $2$, $3$, or $6$. Hence every
vertex of $B(G)$ has order $6$. Such a cyclic subgroup is of the form $PQ$,
where $|P|=2$, $|Q|=3$, and $[P,Q]=1$. Their supports are non-empty and
disjoint, so both are proper subsets of $\{1,2,3\}$; consequently both
$P$ and $Q$ are branching. Thus in this example
\[
    B(G)\cong L(\mathcal A(G)),
\]
where $L(\mathcal{A}(G))$ is the line graph of $A(G)$. 

We first prove that
\[
\diam B(G)\leq 4.
\]
Every edge of $\mathcal A(G)$ has at least one endpoint supported on a
single coordinate. Indeed, the supports of its two endpoints are
non-empty and disjoint subsets of $\{1,2,3\}$, so they cannot both
have size $2$. Now take two arbitrary edges of $\mathcal A(G)$ and
choose a single-coordinate endpoint from each. As observed above, any
two single-coordinate vertices of $\mathcal A(G)$ have distance at
most $3$. Hence the two chosen edges can be joined in the line graph
by a path of length at most $4$. Therefore
\[
\diam B(G)\leq 4.
\]

To prove the reverse inequality, choose edges $PQ$ and $RS$ of
$\mathcal A(G)$ such that
\[
\operatorname{supp}(P)=\{1\},\qquad
\operatorname{supp}(Q)=\{2,3\},
\]
and
\[
\operatorname{supp}(R)=\{2,3\},\qquad
\operatorname{supp}(S)=\{1\},
\]
where $P,R$ have order $2$ and $Q,S$ have order $3$.

We claim that the minimum distance in $\mathcal A(G)$ between an
endpoint of $PQ$ and an endpoint of $RS$ is $3$. Indeed, $P$ and $S$
are not adjacent because they have the same support $\{1\}$. Since
$\mathcal A(G)$ is bipartite and $P,S$ have different prime orders,
their distance cannot be $2$. On the other hand, choosing vertices of
orders $3$ and $2$ supported on $\{2\}$ and $\{3\}$, respectively,
gives a path
\[
P\sim X\sim Y\sim S
\]
of length $3$. The other three pairs of endpoints have distance at
least $3$ as well.

Consequently the two edges $PQ$ and $RS$, regarded as vertices of the
line graph $L(\mathcal A(G))$, have distance $4$. Hence
\[
\diam B(G)\geq 4.
\]
Therefore
\[
\diam B(G)=4.
\]
By Corollary~\ref{cor:D-transfer},
\[
\diam D(G)=4.
\]
\end{example}

Figure~\ref{fig:a4a4a4g} gives a visual picture of the structure described above.
Grouping the single-coordinate branching subgroups according to their
support and prime order, for each coordinate $i$ there are three
order-$2$ subgroups and four order-$3$ subgroups supported on $i$. 
The order-$2$ class supported on $i$ is joined completely to the
order-$3$ class supported on $j$ precisely when $i\ne j$. Thus, at the
level of these six classes, the single-coordinate part has the form of a
$6$-cycle; this accounts for the hexagonal backbone visible in
$\mathcal A(G)$. The outward fan-like parts arise from the branching
subgroups whose support has size $2$, which can be adjacent only to
subgroups of the other prime order supported on the unique omitted
coordinate.

\begin{figure}[ht]
\begin{center}
\includegraphics[scale=0.5]{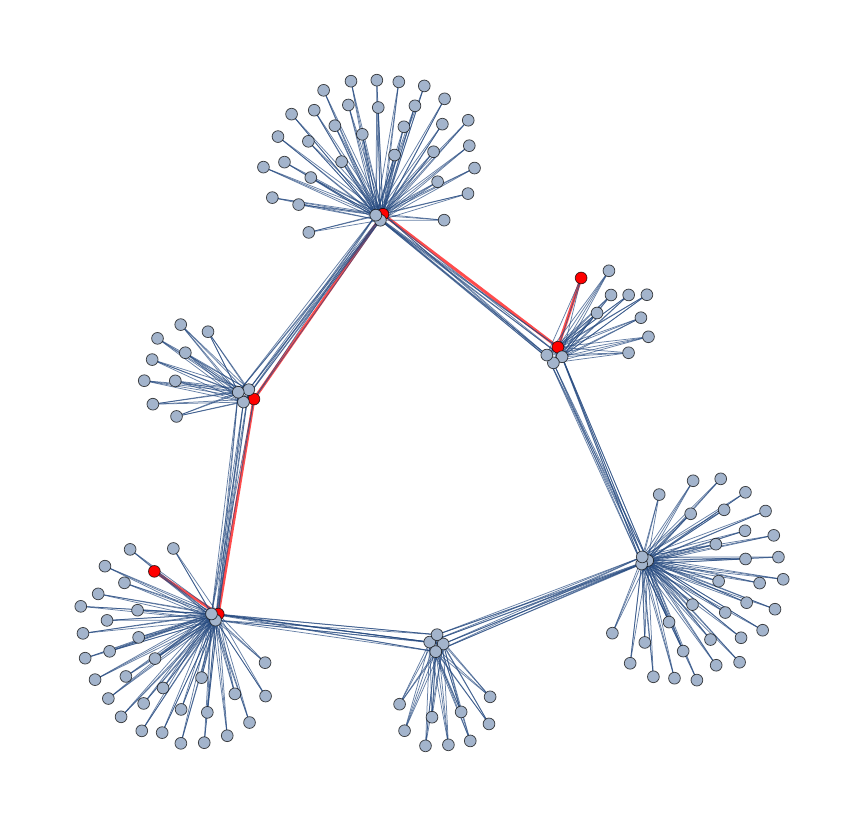} \qquad \includegraphics[scale=0.5]{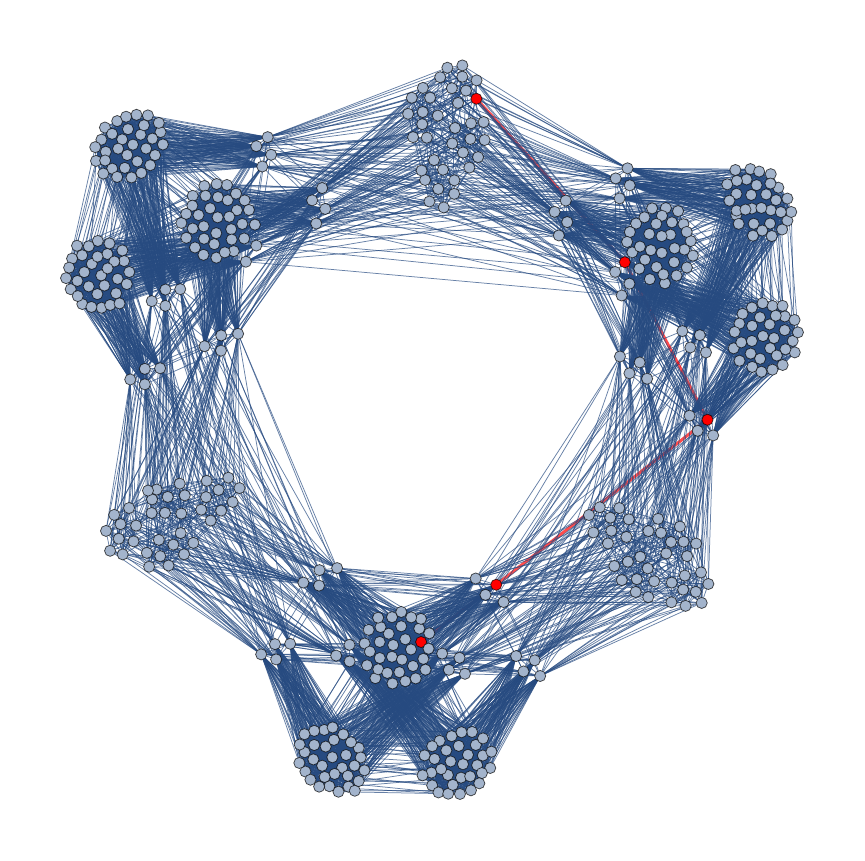}
\caption{The graphs $A(G)$ (left) and $B(G)$ (right) for $G=A_4\times A_4\times A_4$. A diameter path in each graph is highlighted in red.}\label{fig:a4a4a4g}
\end{center}
\end{figure}

There are $9+27=36$ branching subgroups of order $2$ and
$12+96=108$ branching subgroups of order $3$, so
\[
        |V(\mathcal A(G))|=144.
\]
More precisely, the edges consist of six $K_{3,4}$ pieces between
single-coordinate classes, three $K_{3,32}$ pieces, and three $K_{4,9}$
pieces.  Hence
\[
 |E(\mathcal A(G))|
   =6(3\cdot4)+3(3\cdot32)+3(4\cdot9)
   =468.
\]
Since in this example $B(G)\cong L(\mathcal A(G))$, it follows that
\[
        |V(B(G))|=468.
\]
The much denser appearance of $B(G)$ is also explained by the line-graph
construction: the edges incident with a common vertex of $\mathcal A(G)$
become a clique in $B(G)$.

\vspace{1em}
Examples~\ref{ex:a4-square} and~\ref{ex:a4-cube} show that, in the
centerless case, the mixed-centralizer condition~\eqref{eq:MC} is only the
local part of the problem. Global connectedness is governed by the overlap
pattern of the local hubs, which is encoded precisely by the auxiliary graph
$\mathcal A(G)$. Moreover, Example~\ref{ex:a4-cube} shows that the lower
bound
\[
    \operatorname{diam}\mathcal A(G)-1
    \leq
    \operatorname{diam}B(G)
\]
is sharp. The same sharpness is observed for the Mathieu group $M_{11}$.
This is a non-abelian simple group of order
\[
|M_{11}|=7920=2^4\cdot3^2\cdot5\cdot11,
\]
and hence is centerless. A direct computation of the branching prime-order subgroups and their commuting relations
gives
\[
\diam\mathcal A(M_{11})=10,
\qquad
\diam B(M_{11})=9.
\]
Thus the lower bound in Theorem~\ref{thm:centerless-connectivity} is again attained.

\begin{example}[Sharpness of the upper bound]
\label{ex:upper-sharp}
Let
\[
G=S_4\times F,
\qquad
F=C_7\rtimes C_6
 =\langle a,b\mid a^7=b^6=e,\ bab^{-1}=a^3\rangle .
\]
Since $3$ has order $6$ modulo $7$, 
$Z(F)=\set{e}$, and hence \(Z(G)=\set{e}\). Thus $G$ is a (non-simple) centerless group. 

Here $\langle a\rangle\cong C_7$ is normal in $F$, and a subgroup
$K\cong C_6$ satisfying
\[
F=\langle a\rangle K,\qquad \langle a\rangle\cap K=\{e\},
\]
is a complement of $C_7$ in $F$. The complements are conjugate
to $\langle b\rangle$; in particular, there are several of them.

Every prime-order subgroup contained in one of the two direct factors is
branching: two distinct prime-order subgroups of a different order in the
other factor give two distinct cyclic overgroups. We claim that there are no
other branching prime-order subgroups. Indeed, let
\[
U=\langle(x,y)\rangle
\]
be a prime-order subgroup which projects non-trivially onto both factors.
Then $|U|\in\{2,3\}$, since $2$ and $3$ are the only common prime
divisors of $|S_4|$ and $|F|$.
Every non-trivial element $y\in F$ of order $2$ or $3$ lies in 
exactly one complement $K_y\cong C_6$ of $C_7$, and
\(
C_F(y)=K_y
\). 
If $|U|=3$, then $C_{S_4}(x)=\langle x\rangle$, so every cyclic
overgroup of $U$ is contained in $\langle x\rangle\times K_y$ and
hence in the unique cyclic subgroup of order $6$ containing $U$.
If $|U|=2$, the
projection of a cyclic overgroup of $U$ onto $S_4$ cannot have order $4$:
in that case the overgroup would have order divisible by $4$, and its unique
involution would have trivial $F$-coordinate, whereas $U$ has non-trivial
$F$-coordinate. Thus again every cyclic overgroup of $U$ is contained in the
unique cyclic subgroup of order $6$ in $\langle x\rangle\times K_y$ containing
$U$. Hence $U$ is not branching. Therefore the branching prime-order
subgroups of $G$ are precisely those contained in one of the two direct
factors. Moreover, the centralizer of each such subgroup contains the
other direct factor. Hence condition~\eqref{eq:MC} holds.

We now estimate distances in $\mathcal A(G)$. Suppose first that
$P$ and $R$ lie in different direct factors. If $|P|\neq |R|$, then
$P\sim R$. If $|P|=|R|\in\{2,3\}$, choose in the unique complement $K_y\cong C_6$ containing the vertex
from the $F$-factor a subgroup $Q$ of the other order.
Then $Q$ commutes with that vertex, while it
commutes automatically with the vertex from the $S_4$-factor. Hence
\[
d_{\mathcal A(G)}(P,R)\leq 2.
\]

Now suppose that $P$ and $R$ lie in the same factor. If they lie in
$S_4$, a suitable prime-order subgroup of $F$ gives a path of length
at most $2$. The same is true for vertices in $F$, except possibly
when they have orders $2$ and $3$ and lie in different complements.
In that case, say $|P|=2$ and $|R|=3$, let $Q$ be the subgroup of
order $2$ in the complement containing $R$, and choose a subgroup
$T\leq S_4$ of order $3$. Then
\[
P\sim T\sim Q\sim R.
\]
Thus two vertices lying in the same factor are at distance at most
$3$. Consequently
\[
\diam\mathcal A(G)\leq 3.
\]
To see that the bound is attained, put
\[
P=\langle b^3\rangle,
\qquad
Q=\langle (ab)^2\rangle.
\]
The subgroups $P$ and $Q$ have orders $2$ and $3$, respectively, and belong
to different complements of $C_7$. Hence they do not commute. They have no
common neighbour in $\mathcal A(G)$: such a neighbour would have to have
order different from both $2$ and $3$, hence order $7$, but $C_7$ commutes
with neither $P$ nor $Q$. 

Now since
\[
|P|=|\langle(ab)^3\rangle|=2,\qquad
|\langle(1\,2\,3)\rangle|=|Q|=3,
\]
with $P,\langle(ab)^3\rangle,Q\leq F$ and
$\langle(1\,2\,3)\rangle\leq S_4$, subgroups of different prime orders
lying in different direct factors commute. Hence
\[
P\sim\langle(1\,2\,3)\rangle\sim\langle(ab)^3\rangle.
\]
Moreover, $\langle(ab)^3\rangle$ and
$Q=\langle(ab)^2\rangle$ are the subgroups of orders $2$ and $3$,
respectively, in the same complement $\langle ab\rangle\cong C_6$,
and therefore they commute. Thus
\[
P\sim\langle(1\,2\,3)\rangle\sim\langle(ab)^3\rangle\sim Q
\]
is a path of length $3$ in $\mathcal A(G)$. Therefore
\[
\operatorname{diam}\mathcal A(G)=3.
\]

We now exhibit two vertices of $B(G)$ at distance $4$. Recall that
\[
P=\langle(1,b^3)\rangle,\qquad
Q=\langle(1,(ab)^2)\rangle.
\]
Let
\[
\sigma=(1\,2\,3),\qquad \tau=(1\,2),
\]
and define
\[
L_1=\langle(\sigma,b)\rangle,\qquad
L_2=\langle(\tau,ab)\rangle.
\]
Both $L_1$ and $L_2$ are cyclic of order $6$. Moreover,
\[
P=\langle(1,b^3)\rangle\leq L_1,
\qquad
Q=\langle(1,(ab)^2)\rangle\leq L_2,
\]
so $L_1\in\mathcal H(P)$ and $L_2\in\mathcal H(Q)$. Hence, by
Lemma~\ref{lem:local-hub}, both are vertices of $B(G)$.

The unique subgroup of order $3$ in $L_1$ is
\[
\langle(\sigma^2,b^2)\rangle,
\]
which projects non-trivially onto both direct factors, and hence is not
branching. Therefore
\[
S(L_1)=\{P\}.
\]
Similarly, the unique subgroup of order $2$ in $L_2$ is
\[
\langle(\tau,(ab)^3)\rangle,
\]
which also projects non-trivially onto both direct factors and hence is
not branching. Thus
\[
S(L_2)=\{Q\}.
\]

Since we have already shown that
\[
d_{\mathcal A(G)}(P,Q)=3,
\]
it follows that
\[
d_{\mathcal A(G)}\bigl(S(L_1),S(L_2)\bigr)=3.
\]

\begin{figure}[t]
\begin{center}
\includegraphics[scale=0.5]{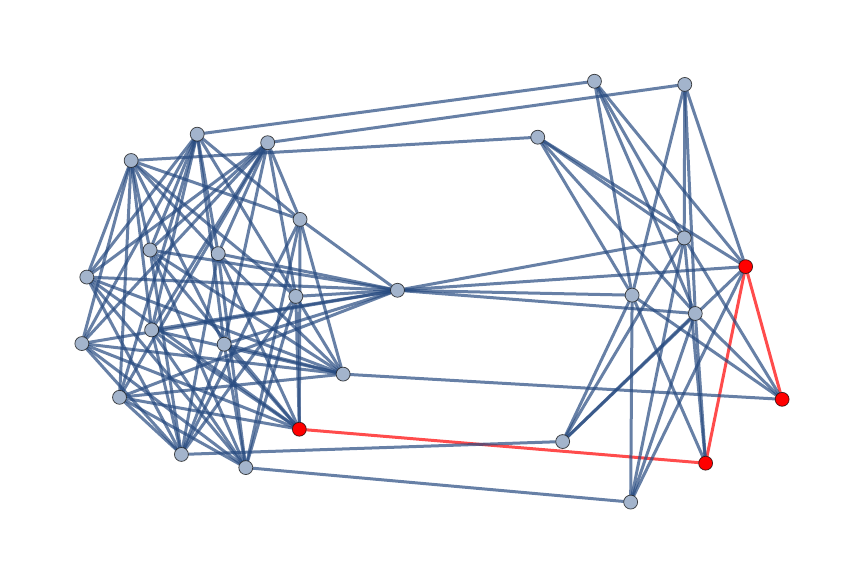} \qquad \includegraphics[scale=0.5]{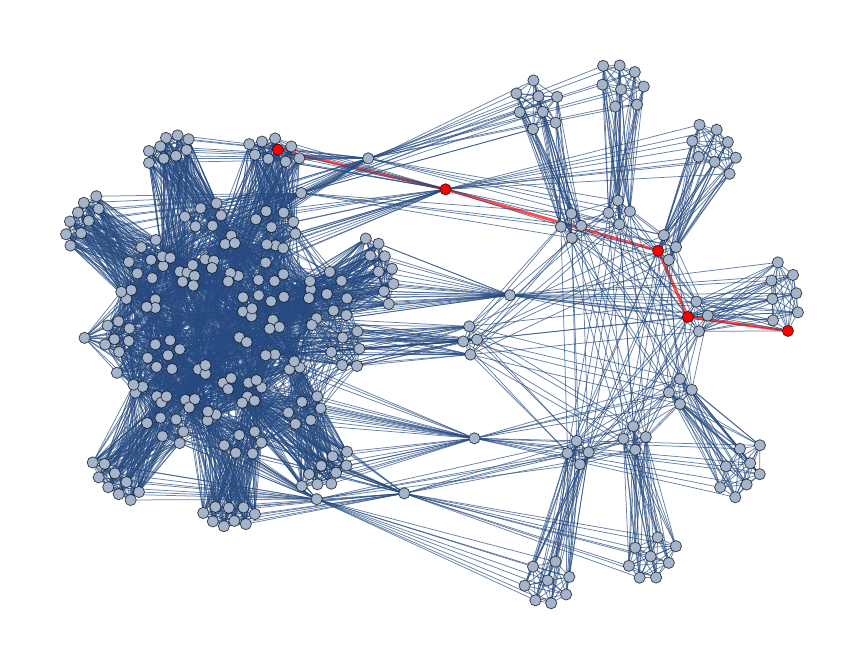}
\caption{The graphs $A(G)$ (left) and $B(G)$ (right) in Example~\ref{ex:upper-sharp}. A diameter path in each graph is highlighted in red.}\label{fig:agbg713}
\end{center}
\end{figure}

Suppose
\[
L_1=X_0\sim X_1\sim\cdots\sim X_m=L_2
\]
is a path in $B(G)$. For each $i\in[m]$, choose a prime-order subgroup
\[
R_i\leq X_{i-1}\cap X_i.
\]
Then $R_i$ is branching, so
\[
R_i\in S(X_{i-1})\cap S(X_i).
\]
Since each $S(X_i)$ is a clique in $\mathcal A(G)$, the sequence
\[
R_1,R_2,\ldots,R_m
\]
determines a walk in $\mathcal A(G)$ of length at most $m-1$ from
$S(L_1)$ to $S(L_2)$. Consequently,
\[
3
=
d_{\mathcal A(G)}\bigl(S(L_1),S(L_2)\bigr)
\leq m-1,
\]
and hence
\[
d_{B(G)}(L_1,L_2)\geq4.
\]

On the other hand, the preceding theorem gives
\[
\operatorname{diam}B(G)
\leq
\max\left\{4,\operatorname{diam}\mathcal A(G)+1\right\}
=4.
\]
Therefore
\[
\boxed{
\operatorname{diam}\mathcal A(G)=3,
\qquad
\operatorname{diam}B(G)=4.
}
\]
Thus the upper bound
\[
\operatorname{diam}B(G)
\leq
\operatorname{diam}\mathcal A(G)+1
\]
is sharp when $\operatorname{diam}\mathcal A(G)\geq3$ (see Figure~\ref{fig:agbg713}).
\end{example}


\section{Concluding remarks}

The generalized-composition viewpoint reduces
the difference graph $D(G)$ to the considerably smaller graph $B(G)$
on cyclic subgroups and reveals the subgroup-theoretic structure
governing its connectedness. Branching subgroups and $b$-normality
provide a common framework for non-emptiness and connectedness, while
central hubs and, in the centerless case, the auxiliary graph
$\mathcal A(G)$ describe how the different local pieces of $B(G)$ are
joined. Together with the established result for cyclic groups, this
gives a complete characterization of non-emptiness and connectedness
of the difference graph for finite groups, and in particular settles
the connectivity part of the open problem of Bera and Cameron for
groups with $\pi(Z(G))\leq1$.

This study of difference graphs not only yields connectedness criteria
and sharp diameter bounds, but also brings to light interesting structural
features of the semilattice of cyclic subgroups in several classes of finite
non-abelian groups. The structural criteria obtained here suggest several natural directions for
further investigation. In particular, one may investigate the exact diameter
of $D(G)$ for standard families of non-abelian finite groups, characterize
the groups attaining the extremal diameter bounds obtained here, and study
more systematically the structure of the auxiliary graph $\mathcal A(G)$
and its relation with the group structure of $G$. In the centerless case, it
would also be interesting to characterize the non-abelian groups $G$ for
which $B(G)\cong L(\mathcal A(G))$.



\end{document}